\documentclass[11pt, a4paper, fleqn]{article}
\usepackage[backend=bibtex]{biblatex}
\usepackage{graphicx}

\usepackage{geometry}
    \usepackage[final,
        protrusion=true,
    ]{microtype}

\DeclareMathSizes{12}{20}{14}{10}

\usepackage{fouriernc}

\usepackage{titlesec}         

\usepackage{marginnote}
\usepackage{amsthm}

\theoremstyle{plain}
\newtheorem{theorem}{Theorem}[subsection]
\newtheorem*{theorem*}{Theorem}
\newtheorem*{condition*}{Condition}
\newtheorem*{definition*}{Definition}
\newtheorem*{corollary*}{Corollary}
\newtheorem{proposition}[theorem]{Proposition}
\newtheorem{lemma}[theorem]{Lemma}
\newtheorem{corollary}[theorem]{Corollary}

\newtheorem{definition}[theorem]{Definition}

\newcounter{introthm}

\theoremstyle{plain}
\newtheorem{introtheorem}[introthm]{Theorem}
\newtheorem{introcorollary}[introthm]{Corollary}
\theoremstyle{definition}
\newtheorem{introremark}[introthm]{Remark}

\theoremstyle{definition}

\newtheorem{remark}[theorem]{Remark}

\newtheorem{example}[theorem]{Example}

\usepackage{tikz-cd} 
\usepackage{amsmath, amssymb, nccmath} 
\usepackage{leftidx}
    \newcommand{\StacksP}[1]{\href{https://stacks.math.columbia.edu/tag/#1}{[#1]}}

    \mathchardef\mhyphen="2D  
    \newcommand{\var}{\rule{0.20cm}{0.20mm}}

    \newcommand{\op}{\ensuremath\mathtt{op}}
    \newcommand{\an}{\ensuremath\mathtt{an}}

    \newcommand{\dd}{\mathtt{d}}

    \newcommand{\CSym}{\operatorname{\mathtt{CSym}}}
    \newcommand{\Sym}{\operatorname{\mathtt{Sym}}}

    \newcommand{\HT}{\mathtt{HT}}
    \newcommand{\cof}{\operatorname{{Cof}}}
    \newcommand{\fib}{\operatorname{{Fib}}}

    \newcommand{\RR}{{\operatorname{\mathit{R}}}}
    \newcommand{\LL}{{\operatorname{\mathit{L}}}}
    \newcommand{\Cot}{\mathrm{L}}
    \newcommand{\Cotan}{\mathrm{L}^\an}

    \newcommand\bb[1]{\mathbf{#1}}

    \newcommand{\rk}{\mathtt{rk}}

    \newcommand{\sh}[1]{\mathcal{#1}}
    \newcommand{\shHom}{\underline{
        \scalebox{1.1}{\ensuremath{\mathtt{Hom}}}}}

    \newcommand{\wtilde}{\widetilde}
    \newcommand{\what}{\widehat}

    \newcommand{\coker}{\mathrm{coker}}
    \newcommand{\HH}{\mathrm{H}}
    \newcommand{\Ext}{\mathrm{Ext}}

    \DeclareMathOperator{\Spa}{{Spa}}
    \DeclareMathOperator{\Spd}{{Spd}}

    \DeclareMathOperator{\Aut}{Aut}
    \DeclareMathOperator*{\colim}{colim}
    
    \DeclareMathOperator{\gr}{\mathtt{gr}}
    \DeclareMathOperator{\End}{{End}}

    \newcommand{\Pic}{\mathrm{Pic}}

    \newcommand{\cont}{\mathtt{cont}}
    \newcommand{\cond}{\mathtt{cond}}
    \newcommand{\et}{\mathtt{\acute{e}t}}
    \newcommand{\fet}{\mathtt{f\acute{e}t}}
    
    \newcommand{\proet}{\mathtt{pro\acute{e}t}}
    \newcommand{\profet}{\mathtt{prof\acute{e}t}}
    \newcommand{\qproet}{\mathtt{qpro\acute{e}t}}
    
    \newcommand{\Dol}{\mathtt{Dol}}
    
    \newcommand{\uni}{\mathtt{uni}}
    
    \newcommand{\luni}{\mathtt{l.uni}}
    \newcommand{\vv}{\mathtt{v}}
    
    \newcommand{\ttt}{\mathtt{t}}

    \newcommand{\Ainf}{\bb{A}_{\mathtt{inf}}}
    \newcommand{\Binf}{\bb{B}_{\mathtt{inf}}}
    \newcommand{\Bdr}{\bb{B}_{\mathtt{dR}}}

    \DeclareMathOperator{\Hom}{Hom}

    \newcommand{\cat}[1]{\mathsf{#1}}
    \newcommand{\icat}[1]{\mathcal{#1}}
    
    \newcommand{\Coh}{\cat{Coh}}
    \newcommand{\DHiggs}{\icat{D}_\mathrm{Higgs}}

        \newcommand{\CAlg}{\cat{CAlg}}

    \newcommand{\Higgs}{\cat{Higgs}}
    \newcommand{\LocSys}{\cat{LocSys}}
    \newcommand{\ILocSys}{\cat{ILocSys}}
    \newcommand{\VB}{\cat{VB}}
    \newcommand{\Rep}{\cat{Rep}}

    \newcommand{\DPerf}{\operatorname{\mathtt{Perf}}}

    \newcommand{\Mod}[1]{{#1}\mhyphen\cat{Mod}}

    \newcommand{\isoto}{\cong}
    
    \newcommand{\reqv}{\xrightarrow{\sim}}
    
    \newcommand{\inj}{\hookrightarrow}
    \newcommand{\surj}{\twoheadrightarrow}

\usepackage{hyperref}

\newcommand{\Perf}{\cat{Perf}}

\title{\textsc{On the unipotent $p$-adic Simpson correspondence}}
\author{Thiago Solovera e Nery}

\begin{document}

\maketitle

\begin{abstract}
  \renewcommand*{\thefootnote}{\fnsymbol{footnote}}
  We prove a (derived) $p$-adic Simpson correspondence for (locally) unipotent
  coefficients on smooth rigid-analytic varieties.  Our results depend on a
  deformation to $\Bdr^+/\xi^2$, and \emph{not} on a choice of an exponential.
  Our methods are higher categorical, hinging on the theory of modules over
  $\bb{E}_\infty$-algebras.  
\end{abstract}

\section*{Introduction}

Our main goal in this paper is to prove a unipotent version of the $p$-adic
Simpson correspondence, which holds for arbitrary smooth rigid-analytic spaces
over a characteristic $0$ algebraically closed, complete field $C$. 
An object in a symmetric monoidal exact category is called \emph{unipotent} 
if it is a successive extension of the unit; an object in a symmetric monoidal
stable $\infty$-category is called \emph{derived unipotent} if it lies in the
smallest stable subcategory containing the unit.
An object is said to be locally unipotent if \'etale-locally on $X$ it is
unipotent.

For a rigid-analytic variety $X$ over $C$, we denote by $\Higgs(X)$ the
category of Higgs bundles on $X$ and by $\DHiggs(X)$ the category of
derived Higgs bundles on $X$ (see Section \ref{sec::derived_Higgs}).
The notation $\DPerf(X_\qproet)$ means perfect complexes of
$\what{\sh{O}}_X$-modules.
Our main result is the following.

\begin{introtheorem}
    [Cor.~\ref{main_theorem_corollary}]
    Let $X$ be a smooth rigid-analytic space over $C$ endowed with a flat
    deformation to $\Bdr^+/\xi^2$.
    Then there are equivalences  
    \[
        \Higgs(X)^\luni \reqv \VB(X_\qproet)^\luni; \quad
        \DHiggs(X)^\luni \reqv \DPerf(X_\qproet)^\luni,
    \]
    of symmetric monoidal exact categories and symmetric monoidal
    stable $\infty$-categories respectively.
    Both restrict to an equivalence on unipotent objects.
    In particular this equivalence also preserves the cohomology of objects on
    both sides.
\end{introtheorem}

Here, $\Bdr^+$ denotes Fontaine's de Rham period ring associated to $C$
(Definition \ref{def::fontaine})
and we are taking deformations along the extension $\Bdr^+ / \xi^2 \to C$.
Any rigid-analytic space which is
defined over a finite extension of $\bb{Q}_p$ will automatically deform
canonically (see Lemma \ref{lemma::canonical_lift}), 
and furthermore, using spreading out techniques of Conrad and Gabber, 
such deformations also exist for proper rigid-analytic spaces over an
algebraically closed complete field $C$, although non-canonically
(cf.~\cite[\S13.2]{BMSI}, \cite[Thm.~7.4.4]{haoyang2019ht}).
Non-canonical deformations also exist for smooth affinoids or 
qcqs curves (Proposition \ref{prop::deformations} and
Corollary \ref{cor::HT_split_affinoid}).

We also note that the derived statement and the classical statement are
essentially equivalent (see proof strategy below for details).
For a definition of derived Higgs bundles we refer to 
Section \ref{sec::derived_Higgs}.
The locally derived unipotent Higgs bundles
include all such bundles with nilpotent Higgs field
(\emph{cf}.~Section \ref{subsection::loc_unipotent_higgs}).

The category of locally unipotent Higgs and quasi-pro-\'etale vector
bundles is not abelian in general, and instead we endow it with the exact
structure induced from locally split short exact sequences of vector bundles 
(respecting the Higgs field).

Finally, when proving the correspondence it is enough, by descent, to prove it
for unipotent objects. When $X$ is proper, the unipotent
correspondence becomes of a more homotopical nature; not only is the category
of \emph{unipotent} Higgs bundles on $X$ now abelian, but it is (neutral)
Tannakian.
In fact, we can answer the unipotent analogue of Question 1.3 of
\cite{heuer2023correspondence}.

\begin{introcorollary}
    [Cor.~\ref{cor::uni_hull}, Def.~\ref{def::uni_fun_group}]
    Let $X$ be as above and assume furthermore that $X$ is proper and connected.
    Fix a geometric point $\bar{x} \to X$ and consider
    its \'etale fundamental group $\pi_1(X,\bar{x})$, endowed with its
    profinite topology.  Then there is an equivalence of categories
    \[
        \Rep_C(\pi_1(X,\bar{x}))^\uni \isoto
        \Higgs(X)^\uni 
    \]
    between unipotent continuous representations of $\pi_1$ on $C$-vector
    spaces and unipotent Higgs bundles on $X$.  This equivalence is canonical
    once a lift of $X$ to $\Bdr^+/\xi^2$ is fixed and the fiber functor 
    of the left-hand side is identified with the functor 
    $E \mapsto E_x$.
\end{introcorollary}

\begin{introremark}
    All results of \S\S1--3 (except Corollary \ref{cor::ht} and those involving
    $\pi_1$), in particular Theorem \ref{main_theorem} and its locally unipotent and
    derived variants, hold verbatim over a perfectoid field $L$ of characteristic
    $0$ containing all $p$-power roots of unity, with $\Bdr^+(L)/\xi^2$ in place of
    $\Bdr^+/\xi^2$.  The hypothesis on roots of unity is used only in the local
    computations of Theorems \ref{thm::local_ht} and \ref{thm::cotangent_ht}. The
    comparison with local systems and $\pi_1$ (Theorems \ref{thm::loc_sys},
    \ref{thm::uni_loc_sys} and their corollaries) relies on the primitive
    comparison theorem and requires $C$ algebraically closed; so does the existence
    of lifts of proper varieties due to Conrad and Gabber
    (Remark \ref{rmk::conrad_gabber}).
\end{introremark}

\subsection*{Strategy of proof}

To prove the main theorem we reinterpret both sides as modules over
an appropriate $\bb{E}_\infty$-algebra, in the sense of \cite{HA}. In simple
terms, this is an object in a derived category of a (sheaf of) rings which
admits an algebra structure whose addition and multiplication laws hold up
to coherent homotopies.
In \cite{HA}, the author constructs symmetric monoidal stable
$\infty$-categories $\Mod{R}$ of modules over any (sheaf of)
$\bb{E}_\infty$-algebras $R$.  
If $R$ is an ordinary commutative ring, then this yields the usual enhancement
of the derived category $D(R)$ of $R$. 

To prove our main theorem, we establish a diagram
\begin{equation*}
\begin{tikzcd}
    \Higgs(X)^\uni \ar[r, "\sim"] \ar[d, hook] & \VB(X_\qproet)^\uni \ar[d, hook] \\
    \Mod{\Sym \wtilde\Omega^1[-1] } \ar[r, "\sim"] & \Mod{\RR\nu_* \what{\sh{O}}_X};
\end{tikzcd}
\end{equation*}
where the vertical arrows are fully-faithful
(Prop.~\ref{prop::bundles_and_modules}, Prop.  \ref{prop::Higgs_A_mod}) and the
lower horizontal arrow is induced by an isomorphism 
\[
    \Psi \colon \Sym \wtilde\Omega^1[-1] \reqv \RR\nu_* \what{\sh{O}}_X,
\]
which can be interpreted as a refined version of the Hodge-Tate
decomposition (\emph{cf}.~Section \ref{sec::splitting_HT}), and is equivalent
to deforming $X$ to $\Bdr^+/\xi^2$.  
The top horizontal arrow will exist and be an equivalence for purely categorical
reasons. 
The derived statement is proven with the same argument.

The object $\Sym \wtilde\Omega^1[-1]$, as the name suggests, is the free
$\bb{E}_\infty$-algebra on the object $\widetilde{\Omega}^1[-1]$. 
It is particularly well
behaved in our setting because we are working over $\bb{Q}_p$, which is of 
characteristic zero (Cor.~\ref{cor::A(X)_sym}). In particular we need not 
worry about the distinction between the free $\bb{E}_\infty$-algebra and
the derived functor of $\Sym$. 

The object $\RR\nu_* \what{\sh{O}}_X$ is the derived pushforward of the unit
$\what{\sh{O}}_X$ of $\sh{D}(X_\qproet)$ to $\sh{D}(X_\et)$. The importance of
the projection $\nu \colon X_\qproet \to X_\et$ to $p$-adic Hodge theory was
one of the main points of \cite{scholze2013adic} who used the Leray spectral
sequence associated to this morphism to deduce the Hodge-Tate decomposition.
The pushforward is an $\bb{E}_\infty$-algebra since $\RR\nu_*$ has a symmetric
monoidal left adjoint.

\subsection*{Relation with other works}

The starting point of non-abelian $p$-adic Hodge theory was Deninger-Werner's
paper on parallel transport of vector bundles on $p$-adic varieties
\cite{deninger2005parallel}, followed by Faltings' paper
\cite{faltings2005simpson}, in which, based on a similar result for complex
varieties established by Corlette, Donaldson, Hitchin and Simpson, a
correspondence is sketched between 
\[
    \{\text{Higgs Bundles on $X$}\} \leftrightarrow 
    \{\text{Generalized representations of $\pi_1(X,x)$} \}
\]
over some fixed pointed connected smooth proper curve. A similar correspondence
was also conjectured to hold for general proper, smooth varieties over some
$p$-adic local field, and a proof is also sketched for \emph{small} objects in
a sense made precise in \cite{faltings2005simpson}.

The $p$-adic Simpson correspondence has seen many recent developments.
The small correspondence, going back to Faltings
\cite{faltings2005simpson} and developed systematically by Abbes, Gros and
Tsuji \cite{abbes2016p}, has been reproved and extended in various settings:
in the good reduction case \cite{wang2023simpson},
or even the semi-stable reduction case \cite{sheng-wang2024semistable},
via prismatic methods in \cite{min-wang2022prismatic},
via the Hodge--Tate stack in
\cite{anschutz-heuer-lebras2023small}; see also Tsuji's theorem
\cite[IV.3.4.16]{abbes2016p}. 

Finally, for smooth proper rigid-analytic varieties, Heuer
\cite{heuer2023correspondence} has established a 
$p$-adic Simpson correspondence for general (not necessarily small)
coefficients, building on his previous work
(e.g.~\cite{heuer2024rankone} for the rank-one case and
\cite{heuer-xu2024curves} for curves).
On the derived side, Ansch\"utz, Heuer and Le Bras
\cite{anschutz-heuer-lebras2023htstacks} construct a derived $p$-adic Simpson
functor for v-perfect complexes via Hodge--Tate stacks.
The dependence of the correspondence on the chosen lift, and its functoriality,
have recently been studied in \cite{pan-yu2026lift} and
\cite{abbes-gros-tsuji2026twisting}.
Our methods, however, are different from all the above.
Again, we emphasize that our correspondence is canonical once a lift to
$\Bdr^+/\xi^2$ has been fixed, no exponential is needed, and it works
for general smooth varieties, not necessarily proper.

\subsection*{Notations and conventions}

We fix once and for all a prime $p$, and a complete algebraically closed
non-archimedean field $C$ of characteristic $0$. 
We use the language of adic spaces, and in particular a \emph{rigid-analytic
variety} over a non-archimedean field $C$ is an adic space $X$ locally of
topologically finite type over $\Spa(C)$.
We will also consider the fundamental group $\pi_1(X,\bar{x})$
of a pointed rigid-analytic variety, always in the sense of finite-\'etale
covers (so, in particular, it is a profinite group).

The free $\bb{E}_\infty$-algebra on some complex $M$ will be
denoted by $\Sym M$.
The free (ordinary) commutative algebra on a module $M$
will be denoted $\CSym M$.
Since we are in characteristic zero one could identify $\Sym M[0] = \CSym M$
for flat $M$, but we will keep the notational difference for clarity.

Our gradings follow the convention that indices on the bottom follow
homological conventions, and indices on top follow cohomological ones.
They are related via $C_n = C^{-n}$.
The cohomology of a complex is denoted by $\sh{H}^n$
to possibly distinguish it from its sheaf (hyper)cohomology.

Given a rigid-analytic variety $X$ over $C$, we denote by
$\Omega^1_X = \Omega^1_{X/C}$ the sheaf of completed differentials
on $X$ (see Appendix \ref{app::cotangent}).
We denote the Tate twists by 
$\wtilde\Omega^1_X = \Omega^1_X (-1)$ and similarly for
$n \in \bb{N}$ we have
$\wtilde\Omega^n_X = \bigwedge^n(\wtilde\Omega^1_X)$
and
$\wtilde T_X = \shHom(\wtilde\Omega^1_X, \sh{O}_{X}) = T_X(1)$.

\subsection*{Acknowledgements}

This paper is a modification of my master thesis at the University of Bonn,
defended on March 2023.

I would like to thank first and foremost my master's advisor Peter Scholze for
suggesting me this wonderful topic, giving me insight of the subject and its
history, and not least in his help of making this document more readable and 
precise.
I would also like to thank Johannes Ansch\"utz, for many hours of discussion and
also helping me with the draft, as well as the insight of how to define derived
Higgs bundles and derived unipotence.

Also thanks to Ben Heuer and Haoyang Guo, whose conversations have
completely shaped the way I think about rigid geometry and
$p$-adic non abelian Hodge theory.  Furthermore, the section relating unipotent
and nilpotent Higgs bundles was motivated by a question from Heuer.  I would
also like to mention Mitya Kubrak, who taught me the characteristic $p$
correspondence of Ogus-Vologodsky, and helped me understand Higgs bundles.
Also thanks to Timo Richarz for comments on the preprint.

The final preparations of this paper were funded through Timo Richarz by the Deutsche
Forschungsgemeinschaft (DFG, German Research Foundation) TRR 326 Geometry and
Arithmetic of Uniformized Structures, project number 444845124, by the European
Research Council (ERC) under Horizon Europe (grant agreement 101040935), the
LOEWE professorship in Algebra, project number LOEWE/4b/519/05/01.002(0004)/87,
and the MSCA doctoral network ReMoLD (project number 101168795) under Horizon
Europe.

\subsection*{AI disclosure}

We have used AI to find typos, mistakes and to help make notation consistent.

\tableofcontents

\section{Pro-\'etale and $\vv$-vector bundles}
In this section, we review the basic properties of the pro-\'etale side 
of the $p$-adic Simpson correspondence.
We define the main objects we are interested in: (unipotent) quasi-pro-\'etale
vector bundles (or, equivalently, $\vv$-bundles).  
We then compare this notion to $C$-local systems and representations of the
fundamental group, and show that they agree on unipotent objects when $X$ is
proper.

In Appendix \ref{app::diamonds} we recall some notions on the theory of
diamonds that will be useful for us.

\subsection{Pro-\'etale vector bundles and local systems}
Any ringed topos comes with a theory of vector bundles and perfect complexes.
Given a rigid-analytic space $X$ over $\Spa C$, we will be interested
in vector bundles and perfect complexes for the sheaves $\underline{C}$, 
and the completed structure sheaf $\what{\sh{O}}_X$
(see Definition \ref{def::completed_structure_sheaf}).  
These objects are related to local systems of $\underline{C}$-modules.

\begin{definition}
    Let $(\icat{T}, \sh{O})$ be a ringed topos.
    We denote by $\VB(\icat{T}, \sh{O})$ (or $\VB(\icat{T})$ when $\sh{O}$
    is understood) the category of $\sh{O}$-vector bundles on $\icat{T}$.
    That is, the full subcategory of $\sh{O}$-modules which are locally
    on $\icat{T}$ isomorphic to a finite free module.
    We denote by $\DPerf(\icat{T}, \sh{O})$ (or $\DPerf(\icat{T})$ when
    $\sh{O}$ is understood) the full subcategory of $\icat{D}(\sh{O})$ given by
    the perfect complexes (see for example \StacksP{08G4}).
    For a rigid-analytic variety $X$ we use the notations
    \[
        \DPerf(X) = \DPerf(X_\an, \sh{O}_X), \quad
        \DPerf(X_\qproet) = \DPerf(X_\qproet, \what{\sh{O}}_X), 
    \]
    and similarly $\VB(X) = \VB(X_\an,\sh{O}_X)$ and
    $\VB(X_\qproet) = \VB(X_\qproet, \what{\sh{O}}_X)$.

    Finally, a $C$-local system on a small $v$-stack $Y$ is a 
    vector bundle for $\underline{C}$ on $Y$.
    The category of such is denoted $\LocSys_C(Y)$.
\end{definition}

We now explain the relation between pro-\'etale vector bundles, $C$-local
systems and representations of the fundamental group. 
Consider a connected rigid-analytic
variety $X$ over $C$ and fix a geometric point $\bar{x} \to X$.  We define the
category (with the usual morphisms)
\begin{alignat*}{4}
    \Rep_{C}(\pi_1(X,\bar{x})) = 
    \Bigg\{\begin{array}{@{}c@{}l}\text{ finite dim.\ continuous $C$-linear}\\
    \text{representations of $\pi_1(X,\bar{x})$} \end{array}\Bigg\}.
\end{alignat*}
We will show now that this symmetric monoidal abelian category can be
identified with the category of local systems on the classifying stack of the
fundamental group.

\begin{definition}
    Let $G$ be a profinite group and $X$ a perfectoid space. A $G$-torsor
    on $X$ is a perfectoid space $Y / X$ with a $\underline{G}_X$-action 
    which is $\proet$-locally equivariantly isomorphic to the trivial
    torsor $G \times X  \to X$. We let $BG$ denote the
    classifying stack of $G$-torsors (over $\Spd C$).
\end{definition}

By  \cite[Lemma 10.13]{scholze2017diamonds}, this is equivalent to the classifying
stack of the group object $\underline{G}$ in the usual topos-theoretic sense.
Being a small $\vv$-stack one can define the quasi-pro-étale site of
$BG$ \cite[Def.~14.1]{scholze2017diamonds}.
Explicitly we say that a morphism of pro-étale stacks $X \to BG$ is
quasi-pro-étale if it is locally separated and the pullback to $S \to BG$ is
pro-étale for all strictly totally disconnected $S$.
The quasi-pro-étale site $BG_\qproet$ is the site of those stacks quasi-pro-étale
over $BG$ with $\vv$-covers.

As usual, there is a map  $\Spd C \to BG$ which classifies the trivial torsor
over $\Spd C$.
Given a diamond $X$ and a morphism $X \to BG$ classifying a torsor $P$ the diagram
\[
\begin{tikzcd}
    P \ar[r] \ar[d] & \Spd C \ar[d] \\
    X \ar[r] & BG
\end{tikzcd}
\]
is cartesian. Hence $\Spd C \to BG$ is quasi-pro-étale in the sense above and also
surjective since $P \to X$ is always surjective
\autocite[Lemma~10.13]{scholze2017diamonds}.  

A standard descent argument now implies the following.

\begin{proposition}
    \label{prop::sheaves_BG}
    Let $G$ be a profinite group. Then there is a canonical equivalence of 
    symmetric monoidal categories
    \[
        \LocSys_C(BG_\qproet) \reqv \Rep_C(G), \quad
        \DPerf(BG_\qproet, C) \reqv \DPerf^{\cond}_C(G)
    \]
    given by the pullback along $\Spd C \to BG$, where the last category 
    denotes the stable subcategory of the derived category of
    condensed $G$-representations whose underlying $C$-module is
    perfect.
\end{proposition}

\begin{remark}
    Again, the derived result is important even if one only cares about objects 
    concentrated in degree zero. It implies the cohomological comparison for finite
    dimensional $C$-representations
\[
    \RR\Gamma_\cond(G,V) \isoto \RR\Gamma(BG, \underline{V})
\]
    where the left hand side is defined as the condensed group cohomology of
    $V$.
    Since $V$ is solid, this is known to be just the usual continuous group
    cohomology of $V$ \cite[Prop.~B.2]{bosco2021drinfeld}.
\end{remark}

We will now construct functors
$\Rep_{C}(\pi_1(X,\bar{x})) \to \LocSys_C(X) \to \VB(X_\qproet)$.
For this we will introduce the pro-\'etale version of the universal covering
space in topology.

\begin{definition}
    [The universal pro-finite-\'etale cover]
    Let $X$ be a connected rigid-analytic variety over $C$, and let $\bar{x}
    \to X$ be a geometric point.  We define the \emph{universal
    pro-finite-\'etale cover} of $X$ to be the limit
    \[
        \wtilde{X} = \lim_{X' \to X} X', \qquad
        X'(C) \ni \bar{x}' \mapsto \bar{x},
    \]
    of all \emph{connected}, pointed, finite \'etale covers $(X',\bar{x}') $
    over $ (X,\bar{x})$.  This limit is taken inside the category of
    $\vv$-sheaves on $\Perf$, and is a locally spatial diamond (being
    an inverse limit of spatial diamonds along qcqs maps).
\end{definition}

By definition, the map $\wtilde{X} \to X$ is a quasi-pro-\'etale (more
precisely pro-finite-\'etale) cover of $X$.  It is a torsor under
$\pi_1(X,\bar{x})$, every finite \'etale cover of $\wtilde{X}$ splits, and any pointed
pro-finite-\'etale cover $(X',\bar{x}') \to (X,\bar{x})$ receives a unique
basepoint preserving map $\wtilde{X} \to X'$ (which is automatically
pro-finite-\'etale).  Before defining the aforementioned functors, we state
some important properties of this covering space.

\begin{proposition}
    \label{prop::univ_cover}
    Let $X$ be a connected, qcqs, pointed, rigid-analytic variety over $\Spa
    C$.  Then $\widetilde{X}$ satisfies the conditions 
    \[
        \HH^0(\wtilde X, \underline C) = C, \quad
        \HH^1(\wtilde X, \underline C) = 0.
    \]
    If $X$ is also proper, then $\HH^0(\wtilde X, \what{\sh{O}}_{X}) = C$.
\end{proposition}

\begin{proof}
    This is analogous to \cite[Prop.~4.9]{heuer2020vbundles} or 
    \cite[Prop.~2.4]{pSimpsonAbeloids} who assume properness, 
    however the proof strategy still works in the qcqs setting.
    We note that we have $\HH^0(\wtilde{X},\sh{O}_{C}/\varpi^n)
    = \sh{O}_{C}/\varpi^n$ which follows from
    \cite[Prop.~14.9]{scholze2017diamonds}, and implies $\HH^0(\wtilde{X}, C)
    \isoto C$.  
    When $X$ is proper, we also get the result on $\what{\sh{O}}_X$ cohomology
    via the primitive comparison theorem with the same argument as
    \cite{heuer2020vbundles}.

    For the result on $\HH^1$, note that we have 
        $\HH^1(\wtilde{X}, \sh{O}_{C}/\varpi^n) = 0$.
    By the coherence of the \'etale topos, taking cohomology commutes with
    filtered colimits; since $\sh{O}_C/\varpi^n$ is
    $p$-torsion, we can reduce to the case of a finite abelian group, where
    this now follows from a similar argument as for $\HH^0$: every
    finite-\'etale torsor on $\wtilde X$ is trivialized already at some
    finite level.
    
    We claim that this implies that $\HH^1(\wtilde{X}, C) = 0$ via an $\RR\lim$
    argument.
    Namely, we have that $\sh{O}_{C} = \RR\lim_n \sh{O}_{C}/\varpi^n$
    by the fact that $X_\qproet$ is replete,
    so there is a short exact sequence
    \[
        0 \to \RR^1\lim_{n}\HH^0(\wtilde X, \sh{O}_{C}/\varpi^n) \to
        \HH^1(\wtilde{X}, \sh{O}_{C}) \to 
        \lim_n \HH^1(\wtilde X, \sh{O}_{C}/\varpi^n) \to 0
    \]
    where the first term vanishes since the 
    projective system in question has surjective transition maps, 
    and the last term vanishes by the argument above.
    The claim now follows by inverting $p$.
\end{proof}

Now we can use $\wtilde{X}$ to build local systems just as we do in topology.
The universal cover $\wtilde{X} \to X$ is classified by a map $l \colon X \to
B\pi_1 = B\pi_1(X,\bar{x})$, and since a finite-dimensional continuous
$C$-representation of $\pi_1(X,\bar{x})$ can be seen as a $C$ local system on
$B\pi_1$, we obtain the first functor via pullback.

\begin{proposition}
    \label{prop::local_systems}
    Let $X$ be a connected, qcqs, pointed rigid-analytic variety over $\Spa C$.
    The pullback map via $l \colon X \to B\pi_1$ defined above determines an
    exact, symmetric monoidal, fully-faithful functor
    \[
        l^* \colon \Rep_C(\pi_1(X,\bar{x})) \inj \LocSys_C(X),
    \]
    whose image consists of all local systems trivialized on a
    pro-finite-\'etale cover of $X$.  Furthermore for a representation $V$ we
    have
    \[
        \HH^0_\cont(\pi_1, V) \reqv \HH^0_\qproet(X,l^*V), \quad
        \HH^1_\cont(\pi_1, V) \reqv \HH^1_\qproet(X,l^*V).
    \]
\end{proposition}

\begin{proof}
    First we identify quasi-pro-\'etale local systems on $B\pi_1(X,\overline{x})$
    via Proposition \ref{prop::sheaves_BG}.
    It is clear from the universal property of $\wtilde{X}$ that a local system
    is trivialized on a pro-finite-\'etale cover of $X$ if and only if it is
    trivial on $\wtilde{X}$.  Therefore we pass to the subcategory
    \[
        \ILocSys_C(X) \subset \LocSys_C(X)
    \]
    of such pro-finite-\'etale local systems and we show that pullback induces
    an equivalence.  Note that the following diagram
    \[
    \begin{tikzcd}
    \wtilde{X} \ar[r] \ar[d] & \Spd C \ar[d] \\
    X \ar[r] & B\pi_1
    \end{tikzcd}
    \]
    is cartesian, which implies that the essential image of the pullback does
    indeed lie in $\ILocSys_C(X)$ since all local systems on $\Spa C$ are
    trivial.

    Now the inverse of the pullback is given by the functor
    \[
        L \mapsto \Gamma(\wtilde{X}, L)
    \]
    which is a finite-dimensional $C$ vector space (by proposition
    \ref{prop::univ_cover}) with a continuous $\pi_1(X,\bar{x})$-action induced
    from the action on $\wtilde{X}$.

    The results about cohomology follow from the \v{C}ech-to-sheaf
    cohomology spectral sequence and the computation $\HH^1(\wtilde{X},C) = 0$
    in Proposition \ref{prop::univ_cover} (since the \v{C}ech cohomology of
    $\wtilde{X}/X$ computes the continuous cohomology of those local systems
    which become trivial on it).
\end{proof}

\begin{remark}
    One can show that $\ILocSys_C(X)$ are the same as local systems $L$ which
    admit a $\sh{O}_{C}$-lattice, that is, a sub-$\sh{O}_{C}$-module $\sh{L}
    \subset L$ such that $\sh{L}$ is a $\sh{O}_{C}$-local system and $L =
    \sh{L}[1/p]$ (cf.~\cite{mannwerner2020local}).
\end{remark}

The map $\LocSys_C(X) \to \VB(X_\qproet)$ is simpler to describe.  It is given
by base change:
\[
    L \mapsto L \otimes_C \what{\sh{O}}_X \in \VB(X_\qproet).
\]
This cannot be fully faithful in general since this would imply 
\[
    C \isoto \Hom(C,C) \to \Hom(\what{\sh{O}}_X, \what{\sh{O}}_X) 
    \isoto \Gamma(X, \what{\sh{O}}_X),
\]
which is not usually an isomorphism.  
However, even if $X$ is proper we can only guarantee that 
\[
    \Hom(L_1, L_2) \reqv 
    \Hom(L_1 \otimes \what{\sh{O}}_X, L_2 \otimes \what{\sh{O}}_X)
\]
is an isomorphism when $L_1$ and $L_2$ are trivialized by some
quasi-pro-\'etale cover $Y \to X$ with $\HH^0(Y, \what{\sh{O}}_X) = C$.  
By Proposition \ref{prop::univ_cover} we obtain a fully faithful functor
$\ILocSys_C(X) \inj \VB(X_\qproet)$.  
The image of such functor consists of quasi-pro-\'etale vector bundles which
are trivial over a pro-finite-\'etale cover of $X$.  
Let $\VB(X)_\profet$ denote the category of such bundles
\footnote{
    Not to be confused with $\VB(X_\profet)$, which are vector bundles in the
    pro-finite-\'etale topology.  These are identified canonically with
    $\Rep_C(\pi_1(X,\bar{x}))$, and therefore only agrees with $\VB(X)_\profet$
    when $X$ is proper by Theorem \ref{thm::loc_sys}.
}.

\begin{theorem}
    [{\cite[Thm.~5.2]{heuer2020vbundles}}] \label{thm::loc_sys} 
    Let $X$ be a proper, connected, rigid analytic space over $C$, with a fixed
    geometric point $\bar{x} \to X$.  
    The functors defined above define exact, symmetric monoidal equivalences
    \[
        \Rep_C(\pi_1(X,\bar{x})) \reqv \ILocSys_C(X) 
        \reqv \VB(X)_\profet.
    \]
\end{theorem}

\begin{remark}
    The equivalence $\ILocSys_C(X) \reqv \VB(X)_\profet$, in contrast with the
    first one, preserves all cohomology groups, and hence can be enriched to an equivalence
    \[
        \DPerf(X_\qproet, C)_\profet \reqv \DPerf(X_\qproet)_\profet,
    \]
    where the subscript $\profet$ means we are considering objects trivialized 
    over a pro-finite-étale cover (equivalently $\wtilde{X}$). 
    One can show this using a generalization of the primitive comparison
    theorem (Theorem \ref{thm::primitive_comparison}) to local systems
    (cf.~\cite[Thm.~5.1]{scholze2013adic} for $\bb{F}_p$-local systems).
    As we will focus on unipotent local systems later on, we will simply deduce
    the cohomological comparison from the case of the unit.
\end{remark}

\begin{remark}
    The above theorem can be thought of as a non-archimedean analogue 
    of the following phenomenon:
    For complex manifolds, the analogous functors (where $\bb{C}$ denotes
    the complex numbers) 
    \[
        \Rep_{\bb{C}}(\pi_1(X, x)) \reqv \LocSys_{\bb{C}}(X) \reqv \VB^\nabla(X_\an)
    \]
    identify all local systems as coming from a representation.
    Furthermore, analytic bundles endowed with a flat connection
    plays the role of the category of quasi-pro-\'etale vector bundles.
\end{remark}

\subsection{Unipotent bundles and $\RR \nu_* \what{\sh{O}}_X$-modules}

We now study unipotent objects in the categories studied in this section.
Since we will also need to do a similar construction for Higgs bundles,
we start with a general categorical definition of unipotence.

\begin{definition}
    \label{def::unipotence}
    Let $\cat{C}$ be an exact category with a fixed object $\bb{1}$, 
    which we call the unit.
    We denote by $\cat{C}^\uni$ the full subcategory of $\cat{C}$ generated by
    successive extensions of the unit.
    We call an object in $\cat{C}^\uni$ unipotent.

    Let $\icat{D}$ be a stable $\infty$-category with a fixed object
    $\bb{1}$, which we call the unit.  We denote by $\icat{D}^\uni
    \subset \icat{D}$ the smallest stable subcategory of $\icat{D}$ that
    contains the unit.  An object of $\icat{D}^\uni$ is called \emph{derived
    unipotent}.\footnote{
        A variation of this notion is that we may define $\icat{D}^\uni$ to be
        the smallest \emph{thick} stable subcategory containing the unit.
        In other words, we also require that unipotent objects are closed 
        under retracts.
        Even though derived unipotent objects in our sense are less general
        than ``thick derived'' unipotent objects, once we \'etale sheafify
        this difference disappears: 
        locally derived unipotent objects are the same as locally derived
        ``thick unipotent'' objects.
        We will work with the simpler version throughout.
    }
\end{definition}

In general, even if $\cat{C} \subset \icat{D}$ is the heart of some
$\ttt$-structure, we cannot guarantee that the notions of derived unipotence
and unipotence agree.  For example, if $\cat{C}$ is the category of finitely
generated $R$ modules, for $R$ a PID, and $\icat{D} = \DPerf(R)$, the only
unipotent $R$-modules are the free ones ($R$ is projective) but \emph{all}
finite modules are derived unipotent since they can be written as a cone
    $M \isoto \cof(R^m \to R^n)$.
However, unipotent objects are always derived unipotent by Lemma  
\ref{lemma::uni_implies_derived_uni}.

We now also consider \emph{derived variants} of the objects
we are considering.  Instead of $\Rep_C(\pi_1(X,\overline{x}))$, $\LocSys_C(X)$ and
$\VB(X_\qproet)$ we can consider
\[
    \DPerf(B\pi_{1,\qproet}, C),~\DPerf(X_\qproet, C),~\DPerf(X_\qproet)
\]
where the perfect complexes are taken in the quasi-pro-\'etale topos. 

\begin{lemma}
    \label{lemma::uni_implies_derived_uni}
    Let $\cat{C}$ be an exact category with unit $\bb{1}$.
    Suppose that $\cat{C} \inj \icat{D}$ is a fully faithful functor
    sending exact sequences to fiber sequences.
    Then
    \[
        \cat{C}^\uni \subset \icat{D}^\uni \cap \cat{C},
    \]
    that is, every unipotent object is derived unipotent.
\end{lemma}

\begin{proof}
    This follows from the fact that if  
    $0 \to E' \to E \to E'' \to 0$ is an extension, then $E = \fib(E'' \to
    E'[1])$. 
\end{proof}

\begin{theorem}
    \label{thm::uni_loc_sys}
    Let $X$ be a proper and connected rigid-analytic variety over $C$.  The
    functors above induce equivalences on unipotent objects
    \[
        \Rep_C(\pi_1(X,\bar{x}))^\uni \reqv
        \LocSys_C(X)^\uni \reqv 
        \VB(X_\qproet)^\uni.
    \]
    For derived unipotent objects, we also have
    $
        \DPerf(X_\qproet,C)^\uni \reqv \DPerf(X_\qproet)^\uni.
    $
\end{theorem}

\begin{proof}
    We note that all functors are symmetric monoidal, since they are induced by
    pullback maps of ringed topoi.  We start by proving that the second arrow
    is fully faithful (classical or derived).  Since all objects are dualizable
    in these categories, we have that
    \[
        \Hom(E,F) = \Hom(\bb{1}, E^\vee \otimes F),
    \]
    so in particular, it is enough to consider maps $\bb{1} \to L$, for a local
    system $L$.  Therefore we reduce the above question to an extension of the
    primitive comparison theorem (Thm.~\ref{thm::primitive_comparison})
    \[
        \RR\Gamma(X,L) \reqv \RR\Gamma(X, L \otimes \what{\sh{O}}_X)
    \]
    to all unipotent coefficients.

    For the classical (non-derived) case, this follows by the base case and an
    induction on the rank.  Namely, we can find an extension $0 \to L' \to L
    \to C \to 0$ which induces a diagram
    \[
    \begin{tikzcd}
        R\Gamma(X, L') \ar[r] \ar[d, "\sim"]&
        R\Gamma(X, L) \ar[r] \ar[d] &
        R\Gamma(X, C)  \ar[d, "\sim"] \\
        R\Gamma(X, L' \otimes \what{\sh{O}}_X) \ar[r]&
        R\Gamma(X, L \otimes \what{\sh{O}}_X) \ar[r]&
        R\Gamma(X, \what{\sh{O}}_X) 
    \end{tikzcd}
    \]
    which implies that $R\Gamma(L) \reqv R\Gamma(L \otimes \what{\sh{O}}_X)$ as
    required.  For the derived case, we just consider the subcategory of all
    objects $\sh{L}$ such that the arrow above is an isomorphism.  We then note
    that $C$ is in this category, and it is closed under shifts and fibers (by a
    variation of the argument above), and hence it is $\DPerf(X_\qproet,C)^\uni$.

    For essential surjectivity in the classical case, it is enough to see that
    if $V$ is a representations of the fundamental group, then any extension of
    $V$ by $\what{\sh{O}}_X$, as a quasi-pro-\'etale vector bundle, comes from
    another representation.  
    This translates to the question of whether the maps
    \[
        \HH^1_\cont(\pi_1(X,\bar{x}), V) \reqv 
        \HH^1(X, L) \reqv 
        \HH^1(X, L \otimes \what{\sh{O}}_X),
    \]
    are isomorphisms, which we know from Proposition \ref{prop::local_systems}
    and the discussion above.

    The only thing left to argue is that the functor $\DPerf(X_\qproet,C)^\uni
    \inj \DPerf(X_\qproet)^\uni$ is essentially surjective.  But that
    follows from the fact that the image contains $\what{\sh{O}}_X$, and
    fully faithfulness and exactness which imply that this image is stable.
\end{proof}

\begin{remark}
    The functor 
    \[
        \DPerf(B\pi_{1,\qproet}, C) \to \DPerf(X_\qproet, C)
    \]
    is \emph{not} fully faithful, even when restricted to unipotent objects.
    In other words, \'etale cohomology is not simply computed as group
    cohomology of the fundamental group.
    When this happens, we could say that $X$ is a $K(\pi,1)$
    for $p$-adic coefficients.
\end{remark}

\begin{corollary}
    \label{cor::pi_1_invariance}
    Let $X$, $Y$ be smooth, proper, connected, pointed
    rigid-analytic varieties over $C$ and suppose that a pointed map
    $f \colon X \to Y$ induces an isomorphism
    \[
        \pi_1(X,\bar{x}) \reqv \pi_1(Y, f\bar{x}).
    \]
    Then it also induces via pullback symmetric monoidal equivalences
    $\VB(Y_\qproet)^\uni \reqv \VB(X_\qproet)^\uni$, and 
    $\LocSys_C(Y)^\uni \reqv \LocSys_C(X)^\uni$ compatible
    with fiber functors at $\overline x$ and $f(\overline x)$.
\end{corollary}

We now finish this section by relating unipotent quasi-pro-\'etale vector 
bundles on $X$ to modules over its derived endomorphism algebra.
Here we recall the existence of a lax symmetric monoidal functor
\[
    \RR\nu_* \colon \DPerf(X_\qproet) \to  
    \icat{D}(X_\et),
\]
which defines a sheaf of $\bb{E}_\infty$-algebras $\RR\nu_* \what{\sh{O}}_X$ on
$X_\et$.

\begin{proposition}
    \label{prop::bundles_and_modules}
    Let $X$ be a rigid-analytic variety.
    The above functor defines a symmetric monoidal equivalence
    \[
        \DPerf(X_\qproet)^\uni \reqv
        \Mod{\RR\nu_*\what{\sh{O}}_X}^\uni,
    \]
    which identifies $\VB(X_\qproet)^\uni$ as the smallest subcategory of
    the right hand side which contains the unit and is closed under
    extensions
    (in the sense of fiber sequences).
\end{proposition}

\begin{proof}
    The proof is similar to the theorem above, but now
    the work lies in showing that the lax structure maps
    \[
        \RR \nu_* \sh{E} \otimes_{\RR\nu_* \what{\sh{O}}_X} 
            \RR \nu_* \sh{E}' \reqv
        \RR \nu_* (\sh{E} \otimes \sh{E}')
    \]
    are equivalences.
    We hence fix $\sh{E}$ and consider the full subcategory 
    of all $\sh{V}$ such that the lax morphism above is an equivalence.
    Then this category contains the unit and is stable since it is closed
    under shifts and fiber sequences:
    if $\sh{V}' \to \sh{V} \to \sh{V}''$ is a fiber sequence then
    \[
    \begin{tikzcd}
        \RR\nu_* \sh{V}' \otimes_{\RR\nu_* \what{\sh{O}}_X} 
            \RR\nu_* \sh{E} \ar[r] \ar[d] & 
        \RR\nu_* \sh{V} \otimes_{\RR\nu_* \what{\sh{O}}_X} \RR\nu_* \sh{E} 
            \ar[d] \ar[r] &
            \RR\nu_* \sh{V}'' \otimes_{\RR\nu_* \what{\sh{O}}_X} 
            \RR\nu_* \sh{E} \ar[d] \\
        \RR\nu_* (\sh{V}' \otimes \sh{E} ) \ar[r] & 
        \RR\nu_* (\sh{V} \otimes \sh{E} ) \ar[r] & 
        \RR\nu_* (\sh{V}'' \otimes \sh{E} )  
    \end{tikzcd}
    \]
    commutes and both rows are fiber sequences, hence if
    two of the vertical arrows are isomorphisms, so is the third.

    Now the proof follows in the same vein as before.
    Namely we use that every object is dualizable to reduce fully 
    faithfulness to the claim that 
    \[
        \RR\Gamma_\et \circ \RR\nu_*  \isoto
        \RR\Gamma ,
    \]
    where the last global sections are taken in the quasi-pro-\'etale site,
    but this is clear.
    Essential surjectivity follows by stability of the image.

    The second claim now follows suit, since any extension in
    $\DPerf(X_\qproet)$ of objects in $\VB(X_\qproet)$ will
    automatically come from a vector bundle
    (take long exact sequence in cohomology, note that the extension is
    concentrated in degree zero and that the category of 
    unipotent vector bundles is closed under extensions).
    We also note that this is a monoidal subcategory, since
    vector bundles are flat.
\end{proof}

To finish our discussion, we briefly mention that unipotence is not
an \'etale-local condition (consider the pullback of a non-trivial 
line bundle in $\Pic(X_\et)$).
To get a local notion, we have to also check unipotence locally.

\begin{definition}
    Let $X$ be a rigid-analytic variety.
    A quasi-pro-\'etale-vector bundle $E \in \VB(X_{\qproet})$ is
    called locally unipotent if there is an \'etale cover
    $Y \to X$ such that $E_Y$ is unipotent.
    A perfect complex $E \in \DPerf(X_\qproet)$ is called 
    locally unipotent if there is $Y$ as above with $E_Y$ derived unipotent.
    Their resulting categories are denoted $\VB(X_\qproet)^\luni$ and
    $\DPerf(X_\qproet)^\luni$ respectively.
\end{definition}

\section{Higgs bundles}
We start by finally defining the other side of the correspondence: Higgs
bundles.  
These objects were first defined on curves by Hitchin
\cite{hitchin1987self}, and generalized to higher dimensional varieties by
Simpson \cite{simpson1992higgs}.

\subsection{Higgs bundles}
In this subsection we define Higgs bundles from a ``hands-on'' perspective, as a
rigid tensor category.  We will later see that Higgs bundles also admit a
derived version, as a symmetric monoidal stable $\infty$-category.


\begin{definition}
    \label{def::higgs}
    Let $X$ be a smooth rigid-analytic variety over $C$, 
    ${E}$ an $\sh{O}_X$-vector bundle on $X_\an$.
    A \emph{Higgs field} on ${E}$ is a global section  
    $\theta \in \Gamma(X,\End({E}) \otimes \wtilde\Omega^1)$
    subject to the condition that
    \begin{equation*}
        \theta \wedge \theta = 0 \quad \text{in } 
        \wtilde\Omega^2 \otimes \End(E).
    \end{equation*}
    A \emph{Higgs bundle} is a vector bundle $E$ endowed with a Higgs field
    $\theta$.
    
    A morphism of Higgs bundle is a morphism of the underlying
    $\sh{O}_{X}$-modules commuting with the Higgs field.
    We denote the category of Higgs bundles by $\Higgs(X)$.
\end{definition}

\begin{remark}
    By \'etale descent of vector bundles on rigid-analytic varieties
    replacing $X_\an$ with $X_\et$ in the definition above yields the
    same result.
\end{remark}

\begin{remark}
    \label{rmk::local_Higgs}
    Locally, $\wtilde{\Omega}^1$ is a free sheaf of rank $d$;
    fix a basis $\omega_1, \dots, \omega_d$ of the twisted differentials 
    and write the Higgs field as 
    $\theta = \sum_i \theta_i \otimes \omega_i$ with 
    $\theta_i \in \Gamma(X, \End(E))$.
    The condition $\theta \wedge \theta = 0$ translates to 
    \[
        [\theta_i, \theta_j] = 0
    \]
    i.e., that these endomorphisms \emph{commute}.  Since $\wtilde\Omega^1$ is
    \emph{dualizable}, we can also write the Higgs field as
    \[
        \theta \colon \wtilde{T}_X \to \End(E),
        \qquad \text{or even } 
        \theta \colon \wtilde{T}_X \otimes {E} \to {E},
    \]
    where $\wtilde{T}_X$ is the twisted tangent bundle.  The condition above
    then translates to commutativity of the image, that is, $[\theta,\theta] =
    0$; or equivalently that it extends to a morphism from the free commutative
    algebra sheaf $\CSym \wtilde{T}_X \to \End(E)$.

    We conclude that Higgs bundles are the same thing as 
    modules over $\CSym \wtilde{T}_X$ whose underlying sheaf
    is a vector bundle.
    We will return to this point later when introducing 
    derived Higgs bundles.
\end{remark}

Given a Higgs bundle $({E}, \theta)$, we can deduce a morphism
\begin{equation*}
    \theta_2\colon {E} \otimes \wtilde\Omega^1 
    \xrightarrow{\theta \otimes 1}
    {E} \otimes \wtilde\Omega^1 \otimes \wtilde\Omega^1 
    \xrightarrow{\wedge}
    {E} \otimes \wtilde\Omega^2
\end{equation*}
using the wedge product on differential forms.  We observe that the condition
$\theta_2 \circ \theta = 0$ is the condition for a Higgs field.  This process
naturally extends to higher forms.

\begin{definition}
    \label{def::dolbeault_cohomology}
    Let $({E}, \theta)$ be a {Higgs bundle}.  We define the \emph{Higgs
    complex} (or \emph{Dolbeault complex}) of $E$ to be
     \begin{equation*}
        \sh{A}(X,{E}) = \left[
        {E} \xrightarrow{\theta} {E} \otimes \wtilde\Omega^1 
        \xrightarrow{\theta_2} {E} \otimes \wtilde\Omega^2
        \xrightarrow{\theta_3} \dots \right]
        \in \icat{D}(X_\an).
    \end{equation*}
    The cohomology of this complex $R\Gamma_{\Dol}(E) =
    R\Gamma(\sh{A}(X,{E}))$ is called the \emph{Dolbeault cohomology} of
    ${E}$.
\end{definition}

\begin{remark}
    A remark on the nomenclature: if $\sh{O}_{X}$ is equipped with the $0$
    Higgs field, then its Dolbeault complex splits
    \begin{equation*}
        \sh{A}(X, \sh{O}_{X}) = 
            \bigoplus_i \wtilde\Omega^i[-i]
    \end{equation*}
    and hence its (hyper)cohomology agrees with the usual definition of
    Dolbeault cohomology (also called the Hodge-cohomology) of $X$.  
    We will later see that this remark is a main ingredient in proving the
    unipotent $p$-adic Simpson correspondence.
\end{remark}

\begin{proposition}
    The category $\Higgs(X)$ admits a canonical closed symmetric monoidal
    structure making the forgetful functor $\Higgs(X) \to \VB(X_\an)$ strong
    monoidal.  This category is then \emph{rigid}: every Higgs bundle is
    dualizable.
\end{proposition}

\begin{proof}
    If $\sh{F}$ and $\sh{G}$ are Higgs bundles 
    with fields both denoted $\theta$, 
    then the tensor product $\sh{F} \otimes \sh{G}$ becomes 
    a Higgs bundle using the Leibniz rule
    \begin{equation*}
        \theta(v \otimes w) = \theta(v) \otimes w + v \otimes \theta(w),
    \end{equation*}
    noting that it squares to zero.
    The unit for this monoidal structure is just the vector bundle
    $\sh{O}_{X}$, with the zero Higgs field.

Similarly, this monoidal structure is closed, with internal hom $\shHom(\sh{F}, \sh{G})$
and underlying Higgs field 
$\theta(f) = \theta f - f \theta$.
Note that 
\begin{equation*}
    \Hom_\Higgs(\sh{O}_{X}, \sh{F}) = 
    \ker (\Gamma(X, \sh{F}) 
        \xrightarrow{\theta} \Gamma(X, \sh{F} \otimes \wtilde\Omega^1) ) =
    \HH^0_{\Dol}(X, \sh{F}).
\end{equation*}

We also note that as usual for closed symmetric monoidal categories,
the dual of some Higgs bundle $E$ has to be 
$   
    E^\vee = \shHom(E, \sh{O}_{X}).
$
That every Higgs bundle is dualizable follows from the
above construction and the rigidity of $\VB(X_\an)$.
\end{proof}

\subsection{Derived Higgs bundles}
\label{sec::derived_Higgs}
We also discuss the notion of a derived  Higgs bundle (called Higgs perfect
complex in \cite{anschutz-heuer-lebras2023htstacks}).  
They are the perfect complexes on the total space of the (twisted) cotangent
bundle of $X$ such that the pushforward $\pi_*$ is perfect on $X$.  
All sheaves and perfect objects are taken with
regards to the \emph{analytic} topology for convenience\footnote{
    But as we will see in Proposition
    \ref{prop::higgs_descent},
    one could in principle define derived Higgs bundles 
    for the \'etale topology instead.
}.

In this section we denote by $H_X = \CSym \wtilde{T}_X$ the classical free
commutative symmetric algebra  on the twisted tangent bundle.
Locally, when $X = \Spa(A,A^+)$, $H_X$ corresponds to $H_A$, the $A$-algebra
\[
    H_A = \CSym_A \wtilde\Omega_A^{\vee} = 
        \bigoplus_{n=0}^\infty (\wtilde\Omega^{\vee \otimes  n}_A)_{\Sigma_n}.
\]
When $\wtilde{\Omega}^1$ is free, this is further isomorphic to 
the polynomial algebra $A[T_1, \dots, T_d]$.
(In particular, this is true when $X$ admits an \'etale map to the torus
$\Spa(C\langle T_1^{\pm}, \dots ,T_d^{\pm} \rangle)$.)

We can see $H_X$ as an analytic sheaf of rings on $X$. Sheaves of $H_X$-modules
can be seen as usual as $H_X$-modules in $\icat{D}(X) = \icat{D}(\sh{O}_{X})$
via its $\sh{O}_{X}$-algebra structure. We have a forgetful functor
\[
    \pi_* \colon \icat{D}(H_X) = \Mod{H_X}(\icat{D}(X)) \to
        \icat{D}(X)     
\]
which we denote by $\pi_*$ for a more geometric intuition. 

\begin{definition}
    Let $X$ be a smooth rigid-analytic variety over $C$.
    The category of \emph{derived Higgs bundles} 
    is defined to be the full subcategory 
    \[
        \DHiggs(X) =  \DPerf(H_X) \times_{\icat{D}(\sh{O}_{X})} 
            \DPerf(\sh{O}_{X})  
        \subset \DPerf(H_X)
    \]
    of $\DPerf(H_X)$ consisting of all objects which are already
    perfect over $\sh{O}_{X}$, meaning all objects $\sh{E}$ such that
    $\pi_* \sh{E}$ lies in $\DPerf(X) \subset \icat{D}(X)$.
    An object in $\DHiggs(X)$ is called a
    \emph{derived Higgs bundle}.
\end{definition}

\begin{remark}
    Since the forgetful functor $\pi_* \colon \DPerf(H_X) \to \icat{D}(X)$ 
    is an exact functor between stable $\infty$-categories,
    $\DHiggs(X)$ is also stable.
    That is, $\DHiggs(X)$ is stable under
    (co)fiber and shifts as a subcategory of $\DPerf(H_X)$.
\end{remark}

Using descent result for perfect complexes we can reduce the theory of Higgs
bundles, locally on $X$, to a purely commutative-algebraic one.
This has been essentially shown in \cite{gregsthesis} for general 
analytic adic spaces.
We can offer an easier proof that works only for rigid-analytic affinoids.

\begin{lemma}\label{lem::perf_affinoid}
    Let $X=\Spa(A,A^\circ)$ be an affinoid rigid-analytic variety. 
    Then
    $P\mapsto \sh O_X\otimes_A P$ induces an equivalence of 
    categories $\DPerf(A)\simeq\DPerf(X,\sh O_X)$.
    The inverse is given by derived global sections.
\end{lemma}

\begin{proof}
    It is clear that $\sh{O}_X \otimes_A^L P$ is perfect over $\sh{O}_X$ if
    $P$ is perfect over $A$.
    Now, $\sh{O}_X$ is even flat over $A$ (\cite[\S7.3.2]{BGR}) so
    this tensor-product is underived.
    This is also fully-faithful:
    indeed by a standard d\'evissage argument we reduce to showing
    that $A = \RR\Hom_A(A,A) \to \RR\Hom_{\sh{O}_X}(\sh{O}_X, \sh{O}_X) = 
    \RR\Gamma(X, \sh{O}_X)$ is an isomorphism, which is 
    Tate's acyclicity theorem \cite[\S8.2]{BGR}.

    Now we will show that if $\sh{E}$ is perfect over $\sh{O}_X$ then the
    counit $\sh{O}_X \otimes_A \RR\Gamma(X, \sh{E}) \reqv \sh{E}$ is an
    isomorphism and $\RR \Gamma(X, \sh{E})$ is a perfect $A$-module.
    The cohomology sheaves $\sh{H}^i(\sh{E})$ are coherent since $\sh{O}_X$ is
    coherent, so by Kiehl's theorem (\cite[\S9.4.3]{BGR}) they are of the form
    $\sh{O}_X \otimes_A M_i$ for $M_i$ finitely generated.
    We conclude by the hypercohomology spectral sequence---which converges 
    because $X$ is quasi-compact---and Tate acyclicity
    that $\HH^i(X,\sh E)=M_i$, which implies that $\RR\Gamma(X, \sh{E})$ is
    pseudo-coherent and that the natural map $\sh O_X\otimes_A\RR\Gamma(X,\sh
    E)\to\sh E$ is an isomorphism (again, by the flatness of $\sh{O}_X$ 
    over $A$ \cite[\S7.3.2]{BGR}). 
    
    It remains to see that $\RR\Gamma(X,\sh E)$ has finite Tor-amplitude.
    Since $A$ is Noetherian and $\RR\Gamma(X,\sh E)$ is pseudo-coherent,
    it suffices to bound the Tor-amplitude of
    $\RR\Gamma(X, \sh{E}) \otimes^{L}_A A/\mathfrak{m}$ uniformly
    in the maximal ideals $\mathfrak{m} \subset A$.
    These correspond to the classical points $x \in X$, and $k(x) = A/\mathfrak{m}$.
    By what we have shown,
    $\RR\Gamma(X, \sh{E}) \otimes^{L}_A k(x)\simeq
    \sh E_x\otimes^{L}_{\sh O_{X,x}}k(x)$,
    and we can bound the tor-dimension uniformly in $x$
    because $\sh E$ is locally strictly perfect and $X$ is quasi-compact.  
\end{proof}

\begin{lemma}
    \label{lemma::fiber_seq}
    Let $R$ be a commutative ring, $P$ a finite projective $R$-module,
    $H = \CSym_R P$, and $M \in \icat{D}(H)$.
    Then $M$ is, naturally in $M$, an iterated cofiber of the $H$-modules
    $H \otimes_R \bigwedge^i P \otimes_R M$, $0 \leq i \leq \rk P$
    (where $H$ acts on the left factor).
    In particular, if $M$ is perfect over $R$, then $M$ is perfect over $H$.
\end{lemma}

\begin{proof}
    For $P = R$ this is \cite[Lemma~6.17(1)]{anschutz-heuer-lebras2023htstacks};
    we recall the argument. Consider the Koszul complex of $H$-bimodules
    \[
        K_\bullet = H \otimes_R \textstyle\bigwedge^\bullet P \otimes_R H, \qquad
        d(a \otimes p_1 \wedge \dots \wedge p_k \otimes b)
        = \sum_j (-1)^{j+1}\, (a p_j \otimes \dots \widehat{p_j} \dots \otimes b
        - a \otimes \dots \widehat{p_j} \dots \otimes p_j b),
    \]
    augmented by the multiplication map $K_0 = H \otimes_R H \to H$.
    It is exact: this can be checked Zariski locally on $R$, where $P$ is free
    with basis $e_1, \dots, e_d$, and then $K_\bullet$ is the Koszul complex of
    the regular sequence $(e_i \otimes 1 - 1 \otimes e_i)_i$ in
    $H \otimes_R H = R[e_i \otimes 1, 1 \otimes e_i]$, whose quotient is $H$.
    The terms of $K_\bullet$ are projective as right $H$-modules, so
    tensoring with $M$ over $H$ gives a finite resolution of $M$ by the
    $H \otimes_R \bigwedge^i P \otimes_R M$, whence the first claim.
    These terms are perfect over $H$ when $M$ is perfect over $R$,
    and perfect objects are stable under cofibers.
\end{proof}

\begin{corollary}
    \label{cor::quasi-coherent_higgs}
    Let $X = \Spa(A,A^\circ)$ be a smooth affinoid rigid-analytic variety
    over $C$.
    Let $\sh E\in\icat D(H_X)$ be such that $\pi_*\sh E\in\DPerf(X)$.
    Then $\RR\Gamma(X,\sh E)\in\DPerf(H_A)$, and the natural map
    $H_X\otimes_{H_A}\RR\Gamma(X,\sh E)\to\sh E$ is an isomorphism.
    In particular $\sh E\in\DPerf(H_X)$, and $\RR\Gamma(X,-)$ and
    $H_X\otimes_{H_A}-$ induce inverse equivalences
    \[
        \DHiggs(X)\simeq \DPerf(H_A)\times_{\icat D(A)}\DPerf(A).
    \]
\end{corollary}

\begin{proof}
    By Lemma~\ref{lem::perf_affinoid}, $M=\RR\Gamma(X,\sh E)$ is a perfect
    $A$-module and the natural map $\sh O_X\otimes_A M\to\pi_*\sh E$ is an
    isomorphism. Since $\wtilde T_X = \sh O_X \otimes_A \wtilde T_A$ and
    $\CSym$ commutes with base change, $H_X=\sh O_X\otimes_A H_A$, so this
    is the underlying map of $H_X\otimes_{H_A}M\to\sh E$, which is therefore
    an isomorphism. Lemma~\ref{lemma::fiber_seq} (applied to
    $P = \wtilde T_A$, which is finite projective since $X$ is smooth) shows
    that $M\in\DPerf(H_A)$, hence $\sh E\simeq H_X\otimes_{H_A}M\in\DPerf(H_X)$.
    Conversely, $H_X \otimes_{H_A} -$ sends $\DPerf(H_A)\times_{\icat D(A)}\DPerf(A)$
    into $\DHiggs(X)$, and the two functors are inverse to each other by the
    above and Lemma~\ref{lem::perf_affinoid}.
\end{proof}

\begin{corollary}
    Let $X$ be a smooth rigid-analytic variety over $C$.
    Then the natural functor 
    \[
        \Higgs(X) \to \DHiggs(X)
    \]
    is fully faithful and agrees with $\VB(X) \inj \DPerf(X)$ on the underlying
    $\sh{O}_{X}$-modules.
\end{corollary}

\begin{proof}
    The data of a Higgs bundle is, as pointed out in Remark
    \ref{rmk::local_Higgs}, just an $H_X$-module $E$ which is a vector bundle
    over $\sh{O}_{X}$.  Hence we need to check that $E$ is perfect over $H_X$.
    (We also remark that there is no reason for $E$ to have tor amplitude $0$ as
    an $H_X$-module.)
    Full faithfulness follows since $\Higgs(X)$ is a full subcategory of
    (the heart of) $\icat{D}(H_X)$.
    Perfectness is a local question, so we may assume $X$ to be affinoid;
    since $\pi_* E$ is a vector bundle, hence perfect,
    Corollary \ref{cor::quasi-coherent_higgs} shows that $E \in \DPerf(H_X)$.
\end{proof}

To pass from derived unipotent objects to locally derived unipotent
objects, we need \'etale descent for derived Higgs bundles.

\begin{proposition}
    \label{prop::higgs_descent}
    The association $X \mapsto \DHiggs(X)$ can be enhanced to a functor from
    $X_\et^\op$ to $\infty$-categories, which is a stack for the \'etale
    topology.
\end{proposition}

\begin{proof}
    For the functoriality, it is enough to show that if
    $f \colon Y \to X$ is an \'etale morphism of smooth rigid varieties
    then the morphism 
    \[
        f^* \colon \DPerf(H_X) \to \DPerf(H_Y)
    \]
    induced by the isomorphism
    $f^*H_X \reqv H_Y$ of $\sh{O}_{Y}$-algebras
    preserves the categories of derived Higgs bundles.
    But this is clear by flat base change $\pi_* f^*\sh{E} \isoto f^* \pi_*
    \sh{E}$ and  the fact that $f^* \colon  \DPerf(X) \to \DPerf(Y)$ preserves
    perfect complexes.

    Analytic descent for $\DHiggs(X)$ is clear by definition.
    Since \'etale morphisms of rigid-analytic varieties are locally
    compositions of open immersions and finite \'etale maps
    \cite[Lemma~2.2.8]{huber2013etale}, we can then
    reduce the general \'etale case to $X = \Spa(A,A^+)$ and $Y =
    \Spa(B,B^+)$ affinoid with $A \to B$ finite \'etale and surjective.
    By Corollary \ref{cor::quasi-coherent_higgs} we reduce to usual
    descent for perfect complexes, since $H_A \to H_A \otimes_A B \isoto H_B$
    is finite \'etale and faithfully flat.
\end{proof}

\begin{remark}
    In usual algebraic geometry, the category of (derived) Higgs bundles
    can be identified as a subcategory of sheaves on the cotangent bundle.
    In rigid geometry this is more subtle as $H_X$ is not an affinoid algebra
    and the cotangent bundle is not even qcqs.
\end{remark}

\subsection{The symmetric monoidal structure and the cohomology of
Higgs bundles}
We have a Hopf algebra structure on $H_X$ coming from the additive group
structure of the cotangent bundle.
Concretely, we use the Hopf-algebra structure on $H_X$ from which we
deduce a coalgebra structure
\[
    H_X \to H_X \otimes_{\sh{O}_X} H_X, \quad
    H_X \to \sh{O}_{X}
\]
which in local coordinates (fixing an \'etale map to a torus with coordinates
$T_i$) are of the form $T_i \mapsto T_i \otimes 1 + 1 \otimes T_i$ and $T_i
\mapsto 0$.  
This endows the category $\icat{D}(H_X)$ of $H_X$-modules
with a ``convolution'' closed symmetric monoidal structure.

Namely, if $E,F$ are $H_X$-modules then $E \otimes F$ is
defined as $E \otimes_{\sh{O}_X} F$, where the $H_X$-module structure
comes from restricting the natural
$H_X \otimes_{\sh{O}_X} H_X$-module structure, along the comultiplication
map defined above.
We will show that this preserves Higgs bundles in the proposition below.

Now we introduce another functor which geometrically corresponds to 
Higgs bundles with support on the zero-section of the cotangent bundle.
Intuitively $s_* \sh{E}$ just endows $\sh{E}$ with the zero Higgs field.

\begin{definition}
    Let $s_* \colon \DPerf(X) \to \DHiggs(X)$ be the restriction of scalars
    along the counit $\epsilon \colon H_X \to \sh{O}_X$.
    In other words, $s_* \sh{E} = \sh{E}$ with $H_X$ acting through $\epsilon$.
\end{definition}

\begin{remark}
    On hearts, this restricts to the fully faithful functor
    $\VB(X) \inj \Higgs(X)$, but we remark that in the derived setting already
    \[
        \Hom_{\DHiggs(X)}(s_* \sh{O}_X, s_* \sh{O}_X) = 
        \RR \Gamma_{\Dol}(X, \sh{O}_X)
    \]
    so the non-triviality of Dolbeault cohomology is an obstruction to $s_*$ being
    fully-faithful.
\end{remark}

\begin{proposition}
    Let $X$ be a smooth rigid-analytic variety over $C$.
    The category of derived Higgs bundles admits a closed symmetric monoidal
    structure, and so do the forgetful and zero-section functors
    \[
        \pi_* \colon \DHiggs(X) \to \DPerf(X), \quad
        s_* \colon \DPerf(X) \to \DHiggs(X).
    \]
\end{proposition}

\begin{proof}
    We must show that the symmetric monoidal structure on $\icat{D}(H_X)$ coming
    from the Hopf algebra structure of $H_X$ preserves $\DHiggs(X)$.
    By definition, the tensor product is perfect over $\sh{O}_X$.
    Corollary \ref{cor::quasi-coherent_higgs} now implies that it is a derived
    Higgs bundle.
    The other claims follow from the construction of the symmetric monoidal
    structure.
\end{proof}

We can now give a more geometric explanation for the functors
$\sh{A}$ and $\RR\Gamma_\Dol$ defined in Definition \ref{def::dolbeault_cohomology}.
As usual, cohomology can be understood as morphisms out of the 
tensor unit.

\begin{definition}
    \label{def::derived_dolbeault_cohomology}
    Let $X$ be a smooth rigid-analytic variety over $C$.
    We define the functors \[
        \sh{A}(X,\var) = \pi_* \RR\shHom_{H_X}(\sh{O}_{X}, \var), \quad
        \RR\Gamma_{\Dol}(X,\sh{E}) = \RR\Gamma \sh{A}(X, \sh{E})
    \]
    from $\DHiggs(X)$ to $\DPerf(X)$ and $\icat{D}(C)$ respectively.
    Here $\RR\shHom_{H_X}(\sh{E}, \sh{V})$ denotes the internal hom 
    computed in $\DPerf(H_X)$.
\end{definition}

\begin{proposition}
    \label{prop::right-adjoint}
    Let $X$ be a smooth rigid-analytic variety over $C$.
    The functor $s_* \colon \DPerf(X) \to \DHiggs(X)$
    has $\sh{A}$ as a right-adjoint, and hence endows $\sh{A}(X, \var)$
    and $\RR \Gamma_{\Dol}$ with a lax-monoidal structure.
    If $E \in \Higgs(X)$ is a classical Higgs bundle on $X$, then the
    definitions \ref{def::derived_dolbeault_cohomology} and
    \ref{def::dolbeault_cohomology} of $\sh{A}$ and $\RR\Gamma_\Dol$ agree.
\end{proposition}

\begin{proof}
    The adjunction part is formal nonsense, but for the computation of
    Dolbeault cohomology one needs a resolution of $\sh{O}_X$.
    Following \cite[Cor.~6.22]{anschutz-heuer-lebras2023htstacks} 
    we can use the Koszul resolution of the zero section $H_X \surj
    \sh{O}_{X}$.  
    This is a resolution
    \[
        \sh{O}_{X} \isoto
        \left[ 
            \bigwedge^d\wtilde{T}_X \otimes_{\sh{O}_{X}} H_X \to \dots 
            \to \wtilde{T}_X \otimes_{\sh{O}_{X}} H_X \xrightarrow{\mu} H_X
        \right]
    \]
    of $\sh{O}_{X}$ by locally free $H_X$-modules.
    Here $\mu$ is the structure map of the $\sh{O}_{X}$-algebra $H_X$.

    We therefore conclude that
    $ \RR\shHom_{H_X}(\sh{O}_{X}, E) \isoto \sh{A}(E, \theta) $
    by identifying 
    $\Hom_{H_X}(\bigwedge^k \wtilde{T}_X \otimes H_X, E) \isoto 
        E \otimes \wtilde{\Omega}^k$.
    The Dolbeault cohomology statement follows by taking derived global
    sections.
\end{proof}

\begin{corollary}
    \label{cor::A(X)_algebra}
    The lax monoidal structure of Proposition~\ref{prop::right-adjoint} endows
    $\sh{A}(X,\sh{O}_{X})$ with the structure of an object in
    $\CAlg(\icat{D}(X))$.
    Its cohomology sheaves are isomorphic to $\sh H^n\sh A(X,\sh O_X)\simeq
    \wtilde\Omega^n$ for $n\ge 0$ (and $0$ for $n<0$), 
    and the associated graded algebra structure is identified with the wedge
    product.
\end{corollary}

\begin{proof}
    The Koszul resolution
    $K=
    \Sym_{\sh O_X}\left(\wtilde T_X[1]
        \xrightarrow{\ \mathrm{id}\ }\wtilde T_X\right) =
        H_X\otimes_{\sh O_X}\wedge^\bullet\wtilde T_X\simeq\sh O_X$
    is a dg Hopf algebra with primitive generators, so its coproduct
    $K\to K\otimes_{\sh O_X}K$ is an $H_X$-linear lift of
    $\sh O_X\isoto\sh O_X\otimes\sh O_X$ for the convolution structure.
    Hence the lax structure on
    $\Hom_{H_X}(K,\sh O_X)=\bigoplus_i\wtilde\Omega^i[-i]$ is dual to the
    shuffle coproduct on $\bigwedge^\bullet\wtilde T_X$, i.e.\ it is the wedge product.
\end{proof}

\subsection{Unipotent Higgs bundles}
In this subsection we analyse the unipotent (and derived unipotent) Higgs bundles,
and their relations to modules over the $\bb{E}_\infty$-algebra
$\sh{A}(X,\sh{O}_{X})$.

\begin{definition}
    \label{def::unipotent_Higgs}
    We denote by $\Higgs(X)^\uni$ the category of 
    unipotent Higgs bundles (cf.~Definition \ref{def::unipotence}),
    and a Higgs bundle in this category is called \emph{unipotent}.
    We denote by $\DHiggs(X)^\uni$ the category of 
    \emph{derived unipotent} Higgs bundles. 
    A derived Higgs bundle in this category is called
    \emph{derived unipotent}.
\end{definition}

\begin{remark}
    A rank one Higgs bundle that is unipotent must be isomorphic
    to the trivial Higgs bundle $\sh{O}_{X}$.
    In this sense, unipotent Higgs bundles are disjoint to Higgs line bundles.
\end{remark}

\begin{proposition}
    \label{prop::unipotent_higgs_monoidal}
    Let $X$ be a smooth rigid-analytic variety over $C$.
    Then the category of (derived) unipotent Higgs bundles inherits 
    the closed symmetric monoidal structure from (derived) Higgs bundles.
\end{proposition}

\begin{proof}
    The unit $\sh{O}_{X}$ is (derived) unipotent.
    Let $E$ be a unipotent Higgs bundle.
    Consider the category of all Higgs bundles $F$ such that
    $E \otimes F$
    is unipotent.
    Then $\sh{O}_{X}$ is in this category and since 
    vector bundles are flat, and unipotent vector bundles are
    closed under extensions,
    this coincides with $\Higgs(X)^\uni$.
    The same argument works to see that the internal Hom is also
    preserved, since $\shHom(E, \var)$ is exact when $E$ is 
    a vector bundle.

    For the derived case a similar argument shows that 
    for any unipotent derived Higgs bundle $\sh{E}$ the category of
    derived Higgs bundles $\sh{F}$ such that $\sh{E} \otimes \sh{F}$
    is derived unipotent is stable, and similarly for the internal Hom.
\end{proof}

\begin{example}
    Let $X$ be a proper, smooth, connected rigid-analytic variety over $C$ with
    $\HH^1(X,C) = 0$ (e.g.~$X$ simply-connected).
    Then the category of unipotent Higgs bundles is trivial in the sense that
    it is equivalent to the category of finite dimensional $C$ vector spaces.

    Indeed, by the Hodge decomposition (Corollary \ref{cor::ht}) we have
    \begin{equation*}
        \HH^1(X, C) \isoto \HH^1(X, \sh{O}_{X}) 
        \oplus \HH^0(X, \wtilde\Omega^1)
    \end{equation*}
    so both groups on the right vanish.  The first vanishing says that all
    vector bundle extensions of $\sh{O}_{X}$ by $\sh{O}_{X}$ are trivial, so if
    $\sh{E}$ is unipotent then it is isomorphic to $\sh{O}_{X}^n$ as a vector
    bundle.  The second implies that the Higgs field is zero, since there are
    no globally defined $1$-forms.
\end{example}

This example shows how unipotent Higgs bundles are intimately linked
with the fundamental group of the variety.
This motivates a Tannakian study of the category of unipotent Higgs bundles.

\begin{proposition}
    \label{prop::tannakian_proposition}
    Let $X$ be a smooth, proper, connected rigid-analytic variety over $C$.
    The category of unipotent Higgs bundles is an abelian, 
    rigid, tensor category.
    Fixing a point $x \in X(C)$, then we also have a canonical
    \emph{fiber functor}
    \begin{equation*}
        F_x \colon \Higgs(X)^\uni \to {C}\mhyphen\cat{Vect}
    \end{equation*}
    sending $\sh{E}$ to $\sh{E}_x$, which makes it into a 
    Tannakian category.
\end{proposition}

Properness is essential for the proposition to work since, for example
a morphism $f \colon \sh{O}_X \to \sh{O}_X$
given by a non-zero divisor which is not a unit is both a monomorphism
and an epimorphism in $\Higgs(X)^\uni$.
The following lemma is an adaptation of Nori's argument 
\cite[Chp.~4,Lemma~2]{nori1982fundamental} to our setting.

\begin{lemma}
    Let $X$ be a smooth, connected rigid-analytic variety over $\Spa C$
    with $\HH^0(X, \sh{O}_X) = C$ (eg. $X$ proper over $\Spa C$).
    Then the category of unipotent Higgs bundles is abelian and the forgetful
    functor from $\Higgs(X)^\uni \to \Coh(X)$ is exact.
\end{lemma}

\begin{proof}
    We show that for every morphism $f \colon E \to F$ of unipotent
    Higgs bundles the kernel and cokernel of $f$, in the abelian category of
    coherent sheaves, are unipotent Higgs bundles for the natural Higgs field
    induced from $E$ and $F$ respectively.

    We proceed by a double induction:
    first assume that $E = \sh{O}_X$ and write $F$ as an extension 
    $0 \to F' \to F \to \sh{O}_X \to 0$.
    The induced map $\sh{O}_X = E \to F \to \sh{O}_X$ is either 
    zero or an isomorphism; 
    in the first case it factors through $F'$ and we proceed by 
    induction, otherwise it is a splitting inducing
    $F \isoto F' \oplus \sh{O}_X$ and the cokernel $F'$ is again
    a Higgs bundle.
    Since the statements are dual, this also shows that
    the kernel of any morphism $E \to \sh{O}_X$ is a unipotent Higgs
    bundle for the induced Higgs field.

    To finish the induction, we write $F$ as an extension 
    $0 \to \sh{O}_X \to F \to F'' \to 0$.
    Consider the induced map $g \colon E \to F''$, which by induction
    has a unipotent kernel and cokernel in the category of 
    Higgs bundles.
    Now if the induced map $\ker g \to \sh{O}_X$ is zero, then
    already $\ker f = \ker g$ and a snake lemma argument implies that
    \[
        0 \to \sh{O}_X \to \coker f \to \coker g \to 0
    \]
    will be short exact in $\VB(X)$, hence in $\Higgs(X)$ 
    and hence $\coker f$ is unipotent.
    If $\ker g \to \sh{O}_X$ is non-zero then by the last paragraph it
    is surjective and $\ker f = \ker (\ker g \to \sh{O}_X)$ and 
    $\coker f$ is unipotent.
\end{proof}

\begin{proof}
    [Proof (of Proposition
    \ref{prop::tannakian_proposition})]
    It follows from the lemma above and Proposition 
    \ref{prop::unipotent_higgs_monoidal} that unipotent Higgs bundles 
    form a rigid abelian tensor-category, so it remains to see that 
    the fiber functor is exact and faithful.  

    Exactness follows since any exact sequence of vector bundles splits on a
    neighbourhood of $x$.
    Finally, faithfulness is clear from the lemma above: namely if a morphism
    $f \colon E \to E'$ is non-zero, its image is now a non-zero vector
    bundle, so $f_x$ is non-zero.
\end{proof}

Using the Tannakian reconstruction theorem we get the following definition,
whose name will be justified in Corollary 
\ref{cor::uni_hull}.

\begin{definition}
    [The unipotent fundamental group]
    \label{def::uni_fun_group} 
    Let $X$ be a smooth, proper, connected rigid-analytic variety over $C$, and
    fix a base point $x \in X(C)$.  
    Then the \emph{unipotent fundamental group} of $X$ at $x$ is defined to be
    the affine group scheme over $C$
    \begin{equation*}
        \pi_1^{\uni}(X,x) = \underline{\Aut}^\otimes(F_x).
    \end{equation*}
    We have a canonical equivalence
    $
        \Higgs(X)^\uni \reqv \Rep_C(\pi_1^{\uni})
    $
    and therefore this affine group is (pro-)unipotent.
\end{definition}

We can now relate unipotent Higgs bundles and $\sh{A}(X,\sh{O}_{X})$-modules.

\begin{proposition}
    \label{prop::Higgs_A_mod}
    The lax-monoidal functor $\sh{A} \colon \DHiggs(X) \to
    \DPerf(X)$ induces a strong monoidal equivalence
    \[
        \DHiggs(X)^\uni \reqv \Mod{\sh{A}(X,\sh{O}_{X})}^\uni,
    \]
    where the right hand side is seen as the stable $\infty$-category of
    modules over the $\bb{E}_\infty$-algebra $\sh{A}(X,\sh{O}_{X})$
    (Proposition \ref{prop::right-adjoint}).

    The full subcategory $\Higgs(X)^\uni \subset \DHiggs(X)^\uni$ is identified
    with the smallest subcategory of the right hand side which contains the
    unit and is closed under extensions.
\end{proposition}

\begin{proof}
    This is essentially the same proof as Proposition
    \ref{prop::bundles_and_modules}.
    Namely, we see that $\sh{A}$ is symmetric monoidal by fixing an $\sh{E}$
    and considering the subcategory of $\sh{V}$ such that the lax structure is
    an isomorphism.  Then the same argument shows this category is 
    stable (since we only use that tensoring and $\sh{A}$ is exact).  Now since
    every object is dualizable, full faithfulness of $\sh{A}$ reduces to the
    isomorphism 
    \[
        \Hom(\sh{O}_{X}, \sh{E}) \isoto 
            \Hom_{\sh{A}(X,\sh{O}_{X})}(\sh{A}(X,\sh{O}_{X}), \sh{A}(X,\sh{E})),
    \]
    which is clear.
\end{proof}

\begin{remark}
    One can show, using the same proof, that
    unipotent Higgs bundles correspond to 
    $\RR\Gamma_{\Dol}(X,\sh{O}_{X})$-modules,
    that is,
    \[
        \RR\Gamma_{\Dol} \colon
    \DHiggs(X)^\uni \reqv \Mod{\RR\Gamma_{\Dol}(X,\sh{O}_{X})}^\uni
    \]
    is a symmetric monoidal equivalence.
    In fact the same methods also show that the functor
    \[
        \RR\Gamma_\et \colon 
            \Mod{\sh{A}(X,\sh{O}_{X})}^\uni \reqv 
            \Mod{\RR\Gamma_{\Dol}(X,\sh{O}_{X})}^\uni
    \]
    is an equivalence.
\end{remark}

\subsection{Locally unipotent Higgs bundles}
\label{subsection::loc_unipotent_higgs}
Unipotent Higgs bundles do not satisfy descent for the \'etale topology; for
example any line bundle, endowed with the zero Higgs field, is locally
unipotent but $\sh{O}_{X}$ is the only unipotent Higgs bundle of rank $1$. 
A more local version of this definition is the following.

\begin{definition}
    \label{def::loc_unipotent_Higgs}
    A Higgs bundle is said to be \emph{locally unipotent} if it is unipotent
    after pullback to an \'etale cover of $X$.  The category of locally
    unipotent Higgs bundles is denoted by $\Higgs(X)^{\luni}$.
    We endow this with the structure of exact category with the 
    same exact sequences as the unipotent case.

    Similarly a derived Higgs bundle $\sh{E}$ is said to be
    \emph{locally derived unipotent} if there exists an \'etale cover
    $f \colon Y \to X$ such that $\sh{E}_Y$ is derived unipotent.
    The full subcategory of $\DHiggs(X)$ consisting of locally
    unipotent bundles is denoted $\DHiggs(X)^\luni$.
\end{definition}

\begin{example}
    \label{ex::coherent_higgs}
    There is a fully faithful functor $ \VB(X_\an) \inj \Higgs(X)^\luni$ 
    given by endowing a vector bundle with the zero field.  This category is again
    usually not abelian, because it contains all vector bundles.
    Any coherent sheaf with zero Higgs
    field is locally derived unipotent, at least locally in the analytic topology.  
    Indeed, by regularity, any coherent sheaf $M$ lies in $\DPerf(X)$, and so we
    can embed it in $\DHiggs(X)$ via $s_*$.
    Passing to an analytic open $U$ where $M$ admits a finite free resolution,
    we see by induction on the length that $s_*M_U$ lies in $\DHiggs(U)^\uni$.
    Hence $s_*$ restricts to a functor $\DPerf(X) \to \DHiggs(X)^\luni$.
\end{example}

\begin{example}
    A class of examples is given by \emph{systems of Hodge bundles} in the sense of
    Simpson \cite{simpson1992higgs}: Higgs bundles of the form
    $E = \bigoplus_p E^p$ with $\theta(E^p) \subset E^{p-1} \otimes
    \wtilde{\Omega}^1_X$ (these are the Higgs bundles fixed, up to
    isomorphism, by the $\bb{G}_m$-action $t (E,\theta) = (E, t\theta)$).
    Indeed, the filtration by $\bigoplus_{q \leq p} E^q$ is preserved by
    $\theta$, with graded pieces carrying the zero Higgs field.
    These Higgs bundles are in particular \emph{nilpotent}; see below for a
    definition and comparison between the notions of local unipotence and
    nilpotence.
\end{example}

\begin{proposition}
    Let $X$ be a smooth rigid-analytic space over $C$.
    The association $Y \mapsto \Higgs(Y)^\luni$
    and $Y \mapsto \DHiggs(Y)^\luni$ is a stack on
    the small \'etale site $X_\et$ of $X$.
    The canonical natural transformations
    \[
        \Higgs(Y)^\uni \to \Higgs(Y)^\luni, \quad
        \DHiggs(Y)^\uni \to \DHiggs(Y)^\luni, 
    \]
    identify the latter as a stackification of the former.
\end{proposition}

\begin{proof}
    By descent for (derived) Higgs bundles, we can see that locally (derived)
    unipotent Higgs bundles form a stack.  Indeed, a descent data along a cover
    $Y \to X$ glues uniquely to a Higgs bundle $E$ and $E_Y$ is locally
    unipotent hence so is $E$.  
    Since $\Higgs(\var)^\uni \subset \Higgs(\var)^{\luni}$ is a full
    sub-prestack and every object of the latter locally comes from the
    former, this is the stackification (and similarly in the derived case).
\end{proof}

Although we will not need the characterizations below, one can relate the notion of
unipotent Higgs bundles with those whose Higgs field is nilpotent.

\begin{definition}
    Let $X$ be a smooth rigid-analytic variety and let $(E, \theta)$ be a Higgs
    bundle over $X$. We say that $E$ is nilpotent if $\theta$ is, in the sense
    that the image of
    \[
        \wtilde T_X \to \End(E)
    \]
    lies in the nilpotent endomorphisms of $E$.
\end{definition}

To check this condition we can locally trivialize
$\wtilde{T}_X$ which yields us $d$ operators
\[
    \theta_1, \dots, \theta_d \colon E \to E
\]
on $E$. Then the definition above says that these operators are nilpotent.
Since the $\theta_i$ commute with each other, this does not depend on a
choice of trivialization.
Alternatively a Higgs bundle is nilpotent if the composite
\[
    E \to E \otimes \wtilde\Omega^1 \to \dots \to  
    E \otimes \wtilde\Omega^1 \otimes \dots \otimes \wtilde\Omega^1
\]
is eventually zero.

Being nilpotent is already local for the étale topology, since one can check
being zero on an étale cover, and the index of nilpotency cannot surpass
the rank of the bundle in question (since $X$ is reduced).
Geometrically these bundles are those set-theoretically supported on the zero
section of the cotangent bundle.

\begin{proposition}
    Let $X$ be a smooth rigid-analytic space and $(E,\theta)$ be a Higgs bundle.
    Then the following hold:
\renewcommand{\theenumi}{\roman{enumi}}
    \begin{enumerate}
        \item[i)] If $E$ is locally unipotent then it is nilpotent. 
        \item[ii)] If $E$ is nilpotent, then it is locally derived unipotent.
        \item[iii)] If $\dim X = 1$ and $E$ is nilpotent, then it is locally unipotent.
    \end{enumerate}
\end{proposition}

\begin{proof}
    For i), we see that the trivial Higgs bundle is nilpotent and an extension
    of nilpotent endomorphisms is nilpotent. For ii) we pass to an open
    affinoid cover trivializing the tangent bundle as above and write
    $\theta_i$ for the nilpotent endomorphisms of $E$. Since they commute we
    see that $K = \bigcap \ker_i \theta_i$ is a nonzero sub-Higgs sheaf
    with zero Higgs field. Passing to the quotient $E/K$ we get another Higgs
    sheaf with nilpotent Higgs field, and this process must end since affinoid
    rings of rigid-analytic varieties are Noetherian.
    Since any Higgs sheaf with zero Higgs field
    is locally derived unipotent (Example \ref{ex::coherent_higgs}),
    so is $E$ by induction.

    For iii), let $E$ be a nilpotent Higgs bundle. Consider as above the subsheaf
    $K = \ker \theta \subset E$.
    We claim that $E/K$ is a vector bundle. 
    By regularity and our dimension assumption we reduce to the case
    where $A$ is a Dedekind domain.
    Since $\theta$ induces an injection $E / K \inj E \otimes \wtilde{\Omega}^1_A$
    into a free module, $E / K$ is torsion-free and hence projective.
    As $K$ is a vector bundle with zero Higgs field, 
    we see by induction that $E$ is locally unipotent.
\end{proof}

\begin{remark}
    Item iii) above is not valid in higher dimensions.
    Take $X = \Spa C\langle T_1, T_2 \rangle$, 
    $E = \sh{O}_X^2$ and \[
        \theta = \begin{bmatrix}
            -T_1T_2 & T_1^2 \\
            -T_2^2 &T_1T_2 
        \end{bmatrix}\dd T_1.
    \]
    Then $\theta$ is nilpotent (its matrix is $v w$ with
    $v = (T_1, T_2)^{\mathsf T}$, $w = (-T_2, T_1)$ and $w v = 0$).
    The kernel is the free rank one submodule
    $\sh{O}_X \cdot v \subset \sh{O}_X^2$, but the quotient
    $E / \ker \theta$ is isomorphic to the ideal sheaf $(T_1, T_2)$, which is
    not free on any neighbourhood of $0$.
    Hence $E$ cannot be unipotent near $0$: the first step $L \subset E$ of a
    unipotent filtration would be a sub-line bundle (with locally free quotient)
    such that $\theta(L) = 0$, hence $L \subset \sh{O}_X \cdot v$, which is
    impossible since $v$ vanishes at $0$.
\end{remark}

\section{The correspondence}
\label{sec::correspondence}
We want to relate unipotent quasi-pro-\'etale vector bundles and unipotent
Higgs bundles.  We have established in Propositions \ref{prop::Higgs_A_mod} and
\ref{prop::bundles_and_modules} that we can relate these to unipotent objects
in the category of $\sh{A}(X,\sh{O}_{X})$-modules and 
$\RR \nu_* \what{\sh{O}}_X$-modules respectively.
We now seek to construct an isomorphism
\[
    \sh{A}(X,\sh{O}_{X}) \reqv \RR\nu_* \what{\sh{O}}_X
\]
of $\bb{E}_\infty$-algebras which will realize the unipotent Simpson
correspondence.

The sheaf of $\bb{E}_\infty$-algebras 
$\RR\nu_* \what{\sh{O}}_X$ admits a filtration, the \emph{Hodge-Tate filtration},
which for $X$ proper, in view of the primitive comparison theorem,
induces a filtration on the \'etale cohomology of $X$. This filtration 
is analogous to the Hodge filtration in complex geometry. 

The proof of the correspondence follows the steps of the Hodge-Tate decomposition,
but also taking the $\bb{E}_\infty$-structure of the objects in consideration.
We proceed by
\begin{itemize}
    \item Showing that the associated graded $\gr_\HT \RR\nu_* \what{\sh{O}}_X$
        is the symmetric $\bb{E}_\infty$-algebra on $\wtilde\Omega^1_X[-1]$
        (Theorems \ref{thm::local_ht} and \ref{thm::cotangent_ht}).
    \item Showing that $\sh{A}(X,\sh{O}_{X})$ is the free $\bb{E}_\infty$-algebra
        on $\wtilde\Omega^1_X[-1]$, and hence is isomorphic to the associated 
        graded of the Hodge-Tate filtration
        (Corollary \ref{cor::A(X)_sym}).
    \item Showing that the Hodge-Tate filtration splits canonically as soon as
        we fix a deformation $\wtilde{X} / (\Bdr^+/\xi^2)$ (Corollary
        \ref{cor::split_ht}), which exists canonically for varieties defined
        over a finite extension of $\bb{Q}_p$, or non-canonically for
        proper $C$-varieties, smooth affinoids and qcqs curves.
\end{itemize}

\subsection{The Hodge-Tate filtration}
In this subsection we define the Hodge-Tate filtration on pro-\'etale cohomology.
We start by showing an important computation of the symmetric algebra on
a flat module in degree $1$.
This will be used in the computation of $\Sym (\wtilde{\Omega}^1[-1])$ and
$\RR\nu_* \what{\sh{O}}_X$.
This is a coconnective analogue of a theorem 
of Illusie (which asserts the same but for $M[1]$ instead)
\cite[Ch.~I, \S4.3.2]{illusie1971thesis}.

\begin{lemma}
    \label{lemma::coconnective_illusie}
    Let $(X, \sh{O}_X)$ be a ringed topos with $\bb{Q} \subset \sh{O}_X$, and
    suppose that $M$ is a flat $\sh{O}_{X}$-module.
    \begin{enumerate}
        \item For every $n \geq 0$ we have an isomorphism 
            $(\bigwedge^n M)[-n] \reqv \Sym^n(M[-1])$.
        \item Let $R \in \CAlg(\icat{D}(\sh{O}_X))$ and 
            $f \colon M[-1] \to R$ a map in $\icat{D}(\sh{O}_X)$
            inducing $f^1 \colon M \to \sh{H}^1(R)$.
            Suppose that $\sh{H}^i(R) = 0$ for $i < 0$ and that,
            for all $n \geq 0$, the map \[
                \bigwedge^n M \to \sh{H}^n(R), \quad
                m_1\wedge \dots \wedge m_n \mapsto f^1(m_1)\dots f^1(m_n)
            \]
            is an isomorphism. 
            Then the map of $\bb{E}_\infty$-algebras
            $\Sym(M[-1]) \to R$ induced by $f$ is an equivalence.
    \end{enumerate}
\end{lemma}

\begin{proof}
    For the first statement, by \cite[Ex.~3.1.3.14]{HA}, 
    we have $\Sym (M[-1]) = \bigoplus_n ((M[-1])^{\otimes n})_{h\Sigma_n}$,
    where the subscript denotes the homotopy coinvariants.
    The action of $\Sigma_n$ on $(M[-1])^{\otimes n} = M^{\otimes n}[-n]$
    is given by permuting the factors \emph{and then twisting by the sign 
    character} by the Koszul sign rule.
    We note that the tensor product is underived since $M$ is flat.

    Now, since we are in characteristic $0$, taking $\Sigma_n$-coinvariants 
    is an exact functor: they are just the image of the idempotent operator
    $1/n! \sum_\sigma \sigma$.
    Hence, the coinvariants are simply the usual coinvariants and 
    therefore $((M[-1])^{\otimes n})_{h\Sigma_n} = (\bigwedge^n M)[-n]$.

    For the second statement, note that $\varphi \colon \Sym(M[-1]) \to R$
    is a map of $\bb{E}_\infty$-algebras, so $\sh{H}^*(\varphi)$ is a 
    map of graded-commutative sheaves of rings which restricts to $f^1$ on
    $\sh{H}^1$. 
    By the first point, $\sh{H}^n(\Sym(M[-1]))$ is generated by products of
    $n$ elements of $\sh{H}^1 = M$.
    Since $\sh{H}^*(\varphi)$ is multiplicative and agrees with $f^1$ in
    degree $1$, it is identified with the maps
    $\bigwedge^n M \to \sh{H}^n(R)$ of the statement, and hence is an
    isomorphism by assumption.
\end{proof}

\begin{remark}
    Applying the above to the strictly graded-commutative algebra
    $\bigoplus_n \bigwedge^n M[-n]$, we see that $\Sym (M[-1])$ is
    \emph{formal}, 
    meaning that it is isomorphic as a sheaf of $\bb{E}_\infty$-algebras to
    the direct sum of its cohomology groups.
\end{remark}

\begin{remark}
    For Lemma \ref{lemma::coconnective_illusie} to be true, it is crucial that
    we are in characteristic $0$.  
    If $\sh{O}_X$ is not a $\bb{Q}$-algebra, then  $(M[-1]^{\otimes
    n})_{h\Sigma_n}$ is not concentrated in a single degree, or even bounded,
    even for $M$ the free rank $1$-module.
\end{remark}

\begin{corollary}
    \label{cor::A(X)_sym}
    Let $X$ be a smooth rigid-analytic variety over $C$.
    The inclusion $\wtilde{\Omega}^1_X[-1] \to \sh{A}(X, \sh{O}_X)$
    induces an equivalence of $\bb{E}_\infty$-algebras
    \[
        \Sym \wtilde\Omega^1_X[-1] \reqv \sh{A}(X,\sh{O}_{X}).
    \]
\end{corollary}

\begin{proof}
    Follows by Lemma \ref{lemma::coconnective_illusie} and Corollary
    \ref{cor::A(X)_algebra}.
\end{proof}

\begin{definition}
    [The Hodge-Tate filtration]
    \label{def::ht_filtration}
    Let $X$ be a smooth rigid-analytic variety over $C$
    and let $\nu \colon X_\qproet \to X_\et$ be the canonical
    map of sites.
    Then the \emph{Hodge-Tate filtration} on $R\nu_* \what{\sh{O}}_X$
    is defined to be the increasing filtration induced by 
    $\tau^{\leqq i}\RR\nu_* \what{\sh{O}}_X$.
    This filtration is multiplicative by the lax monoidality of the
    canonical truncations.
    The associated graded of $R\nu_* \what{\sh{O}}_X$ is
    \[
        \gr_{\HT} \RR\nu_* \what{\sh{O}}_X = 
        \bigoplus_i \RR^i\nu_* \what{\sh{O}}_X[-i]
    \]
    as an $\bb{E}_\infty$-algebra.
\end{definition}

\begin{remark}
    The above definition makes sense for all rigid analytic varieties $X$ over
    $C$, however, it is only well behaved for $X$ smooth.
    One can get a better theory by taking a left Kan extension from
    the smooth case.
    This is made precise using the $\texttt{\'Eh}$ topology 
    in \cite{haoyang2019ht}.
\end{remark}

\begin{remark}
    The above induces a filtration, also called the Hodge-Tate filtration, 
    on the pro-\'etale cohomology 
    $\RR\Gamma_\et(X, \RR\nu_* \what{\sh{O}}_X) 
    = \RR\Gamma_\qproet(X, \what{\sh{O}}_X)$ of $X$.
    This descends to a filtration on cohomology groups
    $\HH^k_\qproet(X, \what{\sh{O}}_X)$ via the image of 
    $\HH^k_\et(X, \tau^{\leqq i} \RR\nu_* \what{\sh{O}}_X)$.
    If the filtration splits, then in fact we have injections
    \[
        \dots \HH^k_\et(X, \tau^{\leqq i} \RR\nu_* \what{\sh{O}}_X)
        \subset \HH^k_\et(X, \tau^{\leqq i + 1} \RR\nu_* \what{\sh{O}}_X)
        \subset \dots \subset \HH^k_\qproet(X, \what{\sh{O}}_X)
    \]
    and the associated graded will be 
    $\HH^k_\et(X, \RR^i \nu_* \what{\sh{O}}_X[-i])
    = \HH^{k-i}_\et(X, \RR^i \nu_* \what{\sh{O}}_X)$.
\end{remark}

We can now compute the associated graded of the Hodge-Tate filtration.  None of the
essential ideas of the proof are new, but we include it here since it is
central to our proof.
We follow closely the expositions \cite{bhatt2017hodge}, \cite{BMSI} and
\cite{haoyang2019ht}.

\begin{theorem}
    \label{thm::local_ht}
    Let $X$ be a smooth rigid-analytic variety over $C$.  Then $R^0\nu_*
    \what{\sh{O}}_X \isoto \sh{O}_{X}$, $R^1 \nu_* \what{\sh{O}}_X$ is a vector
    bundle of rank equal to the dimension of $X$ and the induced map
    \begin{equation*}
        \Sym (\RR^1\nu_*\what{\sh{O}}_X[-1]) \reqv 
        \gr_\HT \RR\nu_* \what{\sh{O}}_X = 
        \bigoplus_n \RR^n \nu_* \what{\sh{O}}_X [-n]
    \end{equation*}
    is an equivalence of $\bb{E}_\infty$-algebras.  Here $\Sym M$ denotes the
    free $\bb{E}_\infty$-algebra on $M \in \icat{D}(X)$.
\end{theorem}

\begin{proof}
    We have a map $\RR^1\nu_* \what{\sh{O}}_X[-1] \to \gr_\HT \RR\nu_*
    \what{\sh{O}}_X$ given by the inclusion of the degree $1$ part.  By 
    Lemma \ref{lemma::coconnective_illusie}, it suffices to show that
    $\RR^0 \nu_* \what{\sh{O}}_X = \sh{O}_X$, that 
    $\RR^1\nu_* \what{\sh{O}}_X$ is locally free and 
    $\bigwedge^n\RR^1\nu_* \what{\sh{O}}_X \to \RR^n \nu_* \what{\sh{O}}_X$ 
    is an isomorphism.

    The statement is \'etale local so we can suppose that $X$
    is an affinoid admitting an \'etale map to a rigid-analytic torus
    \[
        X \to \bb{T}^d = 
        \Spa C\left\langle T_1^{\pm 1}, 
            \dots, T_d^{\pm 1}  \right\rangle.
    \]
    Now, the torus admits a quasi-pro-\'etale cover
    by an affinoid perfectoid
    \[
        {\bb{T}}_\infty^d =      
        \Spa C\left\langle T_1^{\pm 1/p^\infty}, 
        \dots, T_d^{\pm 1/p^\infty}  \right\rangle \to \bb{T}^d
    \]
    and therefore so does $X$ by base change.
    That is, we have a perfectoid affinoid cover
    ${X_\infty} = {\bb{T}_\infty}^d \times_{\bb{T}^d} X \to X$
    of $X$ which, by the acyclicity of $\what{\sh{O}}_X$
    (Theorem \ref{thm::perfectoid_acyclicity}),
    can be used to compute the cohomology of $\what{\sh{O}}_X$
    as \v{C}ech cohomology on global sections.
    Now ${\bb{T}_\infty}^d \to \bb{T}^d$  
    (and hence ${X_\infty} \to X$) is a $\bb{Z}_p(1)^d$ torsor,
    since it is a limit of the 
    $\Spa C\langle T_1^{\pm 1/p^n}, \dots, T_d^{\pm 1/p^n} \rangle$  
    which are $\bb{Z}/p^n(1)^d$-torsors.

    Writing $X = \Spa(R,R^+)$ and 
    ${X_\infty} = \Spa(R_\infty, R^+_\infty)$,
    we conclude that this \v{C}ech cohomology complex is just the 
    continuous cohomology
    \[
        \RR\Gamma_\qproet (X,\what{\sh{O}}_X) \isoto
        \RR\Gamma_\cont(\bb{Z}_p(1)^d, R_\infty)
    \]
    for this action of $\bb{Z}_p(1)^d$ on $R_\infty$.
    This computation is done at integral level in
    \cite[Lemmas~4.5,~5.5]{scholze2013adic}, which implies $\RR^1\nu_*
    \what{\sh{O}}_X$ is free and generates the higher pushforwards via cup
    products.
\end{proof}

\begin{corollary}
    \label{cor::ht_split_tau1}
    The Hodge-Tate filtration on $\RR\nu_* \what{\sh{O}}_X$ splits if and only
    if the induced filtration on $\tau^{\leqq 1} \RR\nu_* \what{\sh{O}}_X$
    splits.
\end{corollary}

\begin{proof}
    Indeed, if we have a splitting $\RR^1\nu_* \what{\sh{O}}_X [-1] \to
    \tau^{\leqq 1} \RR\nu_*\what{\sh{O}}_X$ 
    we can compose with the filtration map to obtain a map
    \[
        \Sym \RR^1\nu_* \what{\sh{O}}_X[-1] \reqv \RR\nu_* \what{\sh{O}}_X
    \]
    which is an equivalence by Lemma \ref{lemma::coconnective_illusie}, whose
    hypotheses are verified in the proof of Theorem \ref{thm::local_ht}.
    Together with Theorem \ref{thm::local_ht}, this gives a splitting.
    The converse is immediate.
\end{proof}

\subsection{The lift to $\Bdr^+/\xi^2$}
The non-abelian Hodge correspondence depends on a flat lift $\wtilde{X}$ of $X$
to the ring of periods $\Bdr^+/\xi^2$, that is, to $\Spa(\Bdr^+/\xi^2)$.
In this subsection we explore a link between the pro-\'etale
cohomology of $\what{\sh{O}}_X$ and the cotangent complex, which finishes our
computation of the associated graded $\gr_\HT \RR\nu_* \what{\sh{O}}_X$ and
allows us to split this filtration when the aforementioned lift exists.

We start by recalling the definition of the rings of periods
$\Ainf$ and $\Bdr^+$.

\begin{definition}
    \label{def::fontaine}
Let $C^\flat$ be the tilt of $C$, hence $\sh{O}_{C^\flat}=\lim_{x\mapsto
x^p}\sh{O}_C/p$, and fix a compatible system
$\varepsilon=(1,\zeta_p,\zeta_{p^2},\dots)\in\sh{O}_{C^\flat}$ of $p$-power
roots of unity. 
Let $\Ainf = W(\sh{O}_{C^\flat})$ and let
$\theta\colon\Ainf\to\sh{O}_C$ be Fontaine's map,
$\theta\bigl(\sum_n p^n[x_n]\bigr)=\sum_n p^n x_n^\sharp$.
It is surjective, and its kernel is generated by the non-zero-divisor
$\xi=1+[\varepsilon^{1/p}]+\dots+[\varepsilon^{1/p}]^{p-1}$.
Put $\Binf=\Ainf[1/p]$ and $\Bdr^+=\lim_n\Binf/\xi^n$; this is a complete
DVR with residue field $C$ and uniformiser $\xi$. We also write
\[
    B_2:=\Bdr^+/\xi^2=(\Ainf/\xi^2)[1/p],
\]
for ease of notation from now on,
regarded as a Tate ring with ring of definition $\Ainf/\xi^2$ (with its
$p$-adic topology).
The map $\theta$ induces a non-split\footnote{
    This splits as an extension of discrete rings, but the splittings
    are not continuous.
} square-zero extension
\[
    0\to C(1)\to B_2\xrightarrow{\ \theta\ }C\to 0,
    \qquad
    \xi B_2=tB_2\cong C\otimes_{\bb{Z}_p}\bb{Z}_p(1),\quad
    t\leftrightarrow 1\otimes\varepsilon,
\]
so that in particular $|\Spa(B_2)|=|\Spa(C)| = \ast$.
\end{definition}

\begin{definition}
    Let $X$ be an adic space over $C$.  A \emph{(flat) lift to $\Bdr^+/\xi^2 =
    B_2$} (or a \emph{deformation to $\Bdr^+/\xi^2$}) is a cartesian square
    \begin{equation*}
    \begin{tikzcd}
        X \ar[r]\ar[d]  & \wtilde{X} \ar[d] \\
        \Spa C \ar[r] & \Spa B_2 
    \end{tikzcd}
    \end{equation*}
    with $\wtilde{X}$ flat over $ B_2$.  
    In other words, a lift to $\Bdr^+/\xi^2$ is an adic space $\wtilde{X}$ flat
    over $\Bdr^+/\xi^2$ together with an identification of $X$ with the zero
    locus of $\xi$.
\end{definition}

The deformation theory of $X$ is controlled by the cotangent complex
$\Cotan_{X/C}$, which is a $p$-complete version of the usual algebraic cotangent
complex (see Definition \ref{def::analytic_cotangent}).

To understand the role of the lift in splitting the Hodge-Tate filtration, we
first remind the reader that if $X \inj Y$ is a closed immersion given by a
coherent ideal sheaf $\sh{I}$, then
$ \tau_{\leqq1}\Cotan_{X/Y} \isoto \sh{I}/\sh{I}^2[1] $
(see Proposition \ref{prop::closed_immersions} on the appendix).
This implies that 
\[
    \tau_{\leqq1}\Cotan_{C/B_2} = C(1)[1], \quad
    \tau_{\leqq1}\Cotan_{X /\wtilde{X}} = \sh{O}_{X}(1)[1],
\]
where the second isomorphism follows from flatness: indeed if $\wtilde X/\Spa
B_2$ is flat, we get an isomorphism $f^* C(1) \isoto \sh{O}_{X}(1) \reqv
\sh{I}$, where $\sh{I}$ is now the ideal defining $X \inj \wtilde{X}$
and $f \colon \wtilde{X} \to \Spa B_2$ is the structure map, by applying
the exact functor $f^*$ to the exact sequence $0 \to C(1) \to \Bdr^+/\xi^2 \to
C \to 0$ seen as sheaves on the topological space $|\Spa C| = |\Spa B_2|$.

Now assume that $X$ is smooth over $C$, so in particular $\Cotan_{X/C} =
\Omega^1_{X/C}$. The transitivity fiber sequence for $X \to \Spa C \to \Spa
\Bdr^+/\xi^2$ now gives us a fiber sequence
\[
    \sh{O}_{X}(1)[1] \to 
    \tau_{\leqq1}\Cotan_{X/B_2} \to 
    \Omega^1_{X/C} \to \sh{O}_X(1)[2].
\]
In particular, there is an obstruction class
$o \in \Ext^2(\Omega^1_{X/C}, \sh{O}_{X}(1)) = \HH^2(X,\wtilde{T}_X)$
which vanishes precisely when this fiber sequence splits. This is equivalent
to constructing an isomorphism
\[
    \tau_{\leqq1}\Cotan_{X/B_2}(-1)[-1] = 
    \sh{O}_{X} \oplus \wtilde\Omega^1_{X/C}[-1]
\]
identifying the map $\sh{O}_{X} \to \tau_{\leqq1}\Cotan_{X/B_2}(-1)[-1]$
with the map from the fiber sequence.  The following theorem now relates the
vanishing of this class and splitting the Hodge-Tate filtration. 

\begin{theorem}
    [{\cite[Prop.~8.15]{BMSI}}
    {\cite[Thm.~7.2.3]{haoyang2019ht}}]
    \label{thm::cotangent_ht}
    Let $X$ be a smooth rigid-analytic space over $C$.  Then there is a
    functorial equivalence
    \[
        \tau_{\leqq1}\Cotan_{X/B_2}(-1)[-1] \reqv 
        \tau^{\leqq 1} \RR\nu_* \what{\sh{O}}_X
    \]
    in the derived category of $\sh{O}_{X}$-modules.  Furthermore, the
    filtration induced by the fiber sequence above agrees with the Hodge-Tate
    filtration on the right hand side.  In particular, $\RR^1\nu_*
    \what{\sh{O}}_X \isoto \wtilde \Omega^1_X$.
\end{theorem}

In particular, this theorem links the obstruction class with $p$-adic Hodge
theory.

\begin{corollary}
    Let $X$ be a smooth rigid-analytic variety over $C$.
    The obstruction class
    $o \in \Ext^2(L_{X/C}^\an, \sh{O}_{X}(1)) = \HH^2(X, \wtilde{T}_X)$
    defined above detects precisely when the Hodge-Tate filtration
    splits.
\end{corollary}

\begin{corollary}
    \label{cor::HT_split_affinoid}
     Let $X/C$ be a smooth affinoid rigid-analytic variety,
     or a smooth qcqs rigid-analytic curve.
     Then the Hodge-Tate filtration always splits (non-canonically).
\end{corollary}

\begin{proof}
     In both cases $\HH^2(X,\wtilde{T}_X) = 0$.
\end{proof}

\subsection{Splitting the Hodge-Tate filtration}
\label{sec::splitting_HT}
Finally, we can relate the problem of lifting $X$ to 
$\Bdr^+/\xi^2$, i.e~the vanishing of the obstruction class $o$ coming from
Theorem \ref{thm::cotangent_ht}, to the problem of splitting of the Hodge-Tate
filtration.  
We are also able to be more precise here: 
each choice of lift $\wtilde{X}  / \Bdr^+/\xi^2$ yields a splitting.
This result is in \cite[Prop.~7.1.4]{haoyang2019ht}, but we offer an
alternative proof.

\begin{theorem}
    \label{thm::bdr_obstruction}
    Let $X$ be a smooth rigid-analytic variety over $C$.
    Then each flat deformation of $X$ to $\Bdr^+/\xi^2$ determines
    a canonical splitting of the Hodge-Tate filtration on 
    $\RR\nu_*\what{\sh{O}}_{X}$.
\end{theorem}

\begin{proof}
    We consider the composition $X \inj \wtilde{X} \to \Spa B_2$ which yields a
    fiber sequence
    \[
        \LL\iota^* \Cotan_{\wtilde{X}/B_2} \to
        \Cotan_{X/B_2} \to
        \Cotan_{{X}/\wtilde X}
    \]
    and since $\iota$ is a closed immersion with ideal 
    $\xi\sh{O}_{\wtilde{X}} \isoto \sh{O}_{X}(1)$ we have
    $\tau_{\leqq 1}\Cotan_{{X}/\wtilde X} = \sh{O}_{X}(1)[1]$.
    We claim that the induced map 
    $\tau_{\leqq 1}\Cotan_{X/B_2}(-1)[-1] \to
    \tau_{\leqq 1}\Cotan_{{X}/\wtilde X}(-1)[-1] \to \sh{O}_{X} $
    splits the inclusion of $\sh{O}_{X} \to \tau_{\leqq 1}\Cotan_{X/B_2}(-1)[-1]$.
    That is, we must show that the composite (which is not a fiber sequence)
    \[
    \begin{tikzcd}
        f^*\Cotan_{C/B_2} \ar[r, "\delta"] &
        \Cotan_{X/B_2} \ar[r, "\alpha"] & 
        \Cotan_{{X}/\wtilde X} 
    \end{tikzcd}
    \]
    ($f$ being the structure morphism $f \colon X \to \Spa C$) induces an
    isomorphism on $\sh{H}_1$.

    Let $\mathcal{H}$ denote the cohomology of complexes.  
    The $\mathcal{H}_1$ are isomorphic to $\sh{O}_{X}(1)$, a
    line bundle, and therefore it is enough to show that the induced morphism
    is surjective.  We see that $\delta$ is an isomorphism from its defining
    fiber sequence. For $\alpha$ we consider the exact sequence
    \[
        \mathcal{H}_1(\Cotan_{X/B_2}) \xrightarrow{\alpha}
        \mathcal{H}_1(\Cotan_{{X}/\wtilde X}) \to
        \mathcal{H}_0(\LL\iota^* \Cotan_{\wtilde{X}/B_2}),
    \]
    and it is enough to show that the last arrow is zero to see surjectivity.
    This is given by the differential
    \[
        \mathcal{H}_1(\Cotan_{{X}/\wtilde X}) \isoto
        \xi \sh{O}_{\wtilde{X}} \xrightarrow{\dd \otimes 1}
        \Omega^1_{\wtilde{X}/B_2} 
        \otimes_{\sh{O}_{\wtilde{X}}} \sh{O}_{X} \isoto
        \mathcal{H}_0(\LL\iota^* \Cotan_{\wtilde{X}/B_2})
    \]
    so it suffices to note that $\dd (\xi g) = \xi \dd g$ vanishes
    modulo $\xi$.
\end{proof}

\begin{corollary}
    \label{cor::split_ht}
    Let $X$ be a smooth rigid-analytic variety over $C$.
    Then any flat lift $X \inj \wtilde{X}$ to $\Bdr^+/\xi^2$
    induces an equivalence of $\bb{E}_\infty$-algebras
    \[
        \Sym \wtilde\Omega^1_X [-1] \reqv
        \RR\nu_* \what{\sh{O}}_X.
    \]
\end{corollary}

\begin{proof}
    Combine Theorem \ref{thm::bdr_obstruction}, Theorem
    \ref{thm::cotangent_ht} (which identifies
    $\RR^1\nu_*\what{\sh{O}}_X \isoto \wtilde\Omega^1_X$) and Corollary
    \ref{cor::ht_split_tau1}.
\end{proof}

We now tackle the problem of actually producing lifts
to $\Bdr^+/\xi^2$, and therefore splittings of the
Hodge-Tate filtration.

\begin{lemma}
    \label{lemma::canonical_lift}
    Suppose that $X$ is defined over $K$, 
    a finite extension of $\bb{Q}_p$.
    Then there is a canonical lift $X \inj \wtilde{X}$ to 
    $\Bdr^+$, and hence to $\Bdr^+/\xi^2$.
\end{lemma}

\begin{proof}
    It suffices to find a \emph{continuous} map making the diagram
    \begin{equation*}
    \begin{tikzcd}
         \Bdr^+/\xi^2 \ar[r] & C \\
                 K \ar[u, dashed] \ar[ur]
    \end{tikzcd}
    \end{equation*}
    commute, since the lift will be given by the base-change from $K$ to
    $\Bdr^+/\xi^2$, which is automatically flat.
    This follows from the fact that $K / \bb{Q}_p$ is \'etale.
\end{proof}

\begin{lemma}
    \label{lemma::etale_lift}
    Let $X,Y$ be  rigid spaces over $C$ and suppose there 
    exists an \'etale map
    \begin{equation*}
        f\colon Y \to X 
    \end{equation*}
    and that $X$ admits a lift $X \inj \wtilde{X}$ to $\Bdr^+/\xi^2$.
    Then $Y$ admits a lift $Y \inj \wtilde{Y}$ to $\Bdr^+/\xi^2$
    which is even \'etale over $\wtilde{X}$.
    This lift is functorial on $Y \in X_\et$.
\end{lemma}

\begin{proof}
    This follows by the topological invariance of the \'etale site, since the
    extension $X \inj \wtilde{X}$ is square-zero.
    That is, given an \'etale morphism $Y \to X$
    there is a \emph{unique} extension $\wtilde{Y} \to \wtilde{X}$
    which is \'etale, and hence flat over $\Bdr^+/\xi^2$.
    Functoriality follows immediately from the uniqueness of the 
    \'etale lift.
\end{proof}

\begin{remark}
    \label{rmk::conrad_gabber}
    Conrad and Gabber have shown that in fact
    all proper rigid spaces $X/C$ (and hence all spaces 
    $Y$ which admit an \'etale map $Y \to X$)
    admit a flat deformation to $\Bdr^+/\xi^2$:
    see \cite[\S13.2]{BMSI} or \cite[Thm. 7.4.4]{haoyang2019ht}.
\end{remark}

\begin{corollary}
    [The Hodge-Tate decomposition]
    \label{cor::ht}
    Let $X$ be a proper, smooth rigid-analytic variety over $C$.
    Then any lift $X \inj \wtilde{X}$ to $\Bdr^+/\xi^2$
    induces a Hodge-Tate decomposition
    \[
        \HH^n(X,C) \isoto \bigoplus_{i+j=n} 
            \HH^i(X,\wtilde\Omega^j_X).
    \]
    If $X$ is defined over $\Spa K$ for $K$ a finite extension of
    $\bb{Q}_p$ and $C = \bb{C}_p$, this is also equivariant for the Galois 
    action on both sides.
\end{corollary}

\begin{proof}
    Apply $\RR\Gamma$ to Corollary \ref{cor::split_ht}
    and use the primitive comparison theorem
    (Theorem \ref{thm::primitive_comparison}).
    The Galois equivariance follows from the canonicity of the
    lift for varieties defined over $K$.
\end{proof}

\subsection{Finishing the proof}
Now we have all the ingredients for proving the unipotent $p$-adic Simpson
correspondence.  We state several versions of it.

\begin{theorem}
    [The unipotent correspondence]
    \label{main_theorem}
    Let $X$ be a smooth rigid-analytic space over $C$, endowed with a lift
    $\wtilde{X}$ to $\Bdr^+/\xi^2$.  There is a canonical equivalence of
    symmetric monoidal exact categories
    \[
        \Higgs(X)^\uni \reqv \VB(X_\qproet)^\uni
    \]
    between unipotent Higgs bundles and unipotent quasi-pro-\'etale vector
    bundles. 

    Similarly, under the same assumptions, there is a canonical equivalence
    of symmetric monoidal stable infinity categories
    \[
        \DHiggs(X)^\uni \reqv \DPerf(X_\qproet)^\uni
    \]
    between derived unipotent Higgs bundles and derived unipotent
    perfect complexes on $X_\qproet$. Both equivalences are natural for
    morphisms which can be lifted to the deformations.
\end{theorem}

\begin{proof}
    Follows from Corollary \ref{cor::split_ht} together with Propositions
    \ref{prop::bundles_and_modules} and \ref{prop::Higgs_A_mod}
    by noting that both the $1$-categorical and derived unipotent objects are
    preserved under monoidal categorical equivalences, and that
    \[
        \Sym \wtilde\Omega^1_X[-1] \reqv \sh{A}(X,\sh{O}_{X}) 
    \]
    as $\bb{E}_\infty$-algebras
    (Corollary \ref{cor::A(X)_sym}).
\end{proof}

\begin{corollary}
    \label{main_theorem_corollary}
    Let $X$ be as above.
    There are symmetric monoidal equivalences
    \[
        \Higgs(X)^\luni \reqv \VB(X_\qproet)^\luni, \quad
        \DHiggs(X)^\luni \reqv \DPerf(X_\qproet)^\luni,
    \]
    between \emph{locally} unipotent objects in each category.
\end{corollary}

\begin{proof}
    If we fix a lift of $X$ to $\Bdr^+/\xi^2$, then we also get a lift of each
    $Y \in X_\et$ in a compatible way by Lemma \ref{lemma::etale_lift}.  It
    follows that the correspondence above can be improved to an equivalence of
    pre-stacks on $X_\et$, and hence an equivalence on their sheafifications.
\end{proof}

\begin{corollary}
    \label{cor::cohomology}
    Let $X$ be a smooth rigid-analytic $C$-variety which admits a deformation
    to $\Bdr^+/\xi^2$. If $(E,\theta)$ is a locally unipotent Higgs bundle on
    $X$, and $\sh{V}$ is the corresponding quasi-pro-\'etale vector bundle,
    then
    \[
        \RR\Gamma_\qproet(X,\sh{V})
        \isoto \RR\Gamma_\Dol(X, E).
    \]
\end{corollary}

\begin{proof}
    Both sides are the mapping spectra out of the unit,
    $\RR\Gamma_\Dol(X,E) = \Hom_{\DHiggs(X)}(\sh{O}_X, E)$ and
    $\RR\Gamma_\qproet(X,\sh{V}) = \Hom_{\DPerf(X_\qproet)}(\what{\sh{O}}_X,\sh{V})$,
    and the equivalence of Corollary \ref{main_theorem_corollary} is symmetric
    monoidal, so it preserves the unit and mapping spectra.
\end{proof}

\begin{corollary}
    Let $f \colon X \to Y$ be a morphism of smooth, \emph{proper},
    connected rigid-analytic varieties which is compatible with
    deformations to $\Bdr^+/\xi^2$ (e.g. if they are defined over
    a $p$-adic field).
    Let also  $\bar{x}$ be a geometric point of $X$ and
    $\bar{y} = f(\bar{x})$ and suppose that
    $f_* \colon \pi_1(X,\bar{x}) \reqv \pi_1(Y,\bar{y})$ is an 
    isomorphism.
    Then we have a symmetric monoidal equivalence of categories
    \[
        \Higgs(Y)^\uni \reqv \Higgs(X)^\uni.
    \]
    That is, unipotent Higgs bundles are invariant
    under $\pi_1$-equivalences.
\end{corollary}

\begin{proof}   
    Follows by Corollary \ref{cor::pi_1_invariance} and the equivalence
    above, noting that it (the equivalence) preserves fiber functors.
\end{proof}

\begin{remark}
    The derived version of this statement does not hold,
    as one sees already for $\bb{P}^1 \to \Spa C$.
\end{remark}

In the same vein, we can now recognise the Tannakian group
associated to $\Higgs(X)^\uni$ as \textit{le $\pi_1$ rendu nilpotent}
    \cite{droite}
    (see also \cite[\S 10.25]{droite} for the transcendental version
    of this corollary.)

\begin{corollary}
    \label{cor::uni_hull}
    Let $X$ be proper, smooth connected and pointed over $C$.
    The unipotent fundamental group
    (Definition \ref{def::uni_fun_group}) is the unipotent hull of
    $\pi_1(X,\bar{x})$.
\end{corollary}

\begin{proof}
    By Remark \ref{rmk::conrad_gabber}, $X$ admits a lift to $\Bdr^+/\xi^2$.
    Fixing one, Theorems \ref{thm::uni_loc_sys} and \ref{main_theorem} give
    a symmetric monoidal equivalence
    $\Rep_C(\pi_1(X,\bar{x}))^\uni \simeq \Higgs(X)^\uni$ compatible with
    the fiber functors. Hence $\pi_1^\uni(X,x)$ is the Tannakian group of
    $\Rep_C(\pi_1(X,\bar{x}))^\uni$, which is the unipotent hull of
    $\pi_1(X,\bar{x})$ essentially by definition, see
    \cite[Eq.~10.24.2]{droite}.
\end{proof}

\appendix

\section{The cotangent complex and deformation theory}
\label{app::cotangent}
Adic spaces admit a theory of cotangent sheaves which are defined analogously
to schemes, but taking into consideration the topology of our modules.
Similarly, they admit a theory of analytic cotangent complexes which
control deformations as adic spaces.
We recall the definitions and the properties used in the main text, and refer
to \cite[Sec.~1.6]{huber2013etale}, \cite[\S7.2]{AlmostRingTheory} and
\cite[\S7.1]{haoyang2019ht} for proofs.
In this appendix the analytic cotangent sheaf carries an ornament
$\Omega^\an$ for clarity, but we drop it in the main text, since there we only
deal with adic spaces.

\begin{definition}
    [Huber {\cite[Def.~1.6.1]{huber2013etale}}]
    Let $A \to B$ be a morphism of Huber rings which is topologically of finite
    type. A \emph{universal derivation} is a continuous derivation
    \[
        \dd \colon B \to \Omega_{B/A}^\an
    \]
    into a complete topological $B$-module $\Omega_{B/A}^\an$ which is
    universal among such, i.e., every continuous $A$-linear derivation of $B$
    into a complete $B$-module $M$ factors uniquely through a continuous
    $B$-linear map $\Omega_{B/A}^\an \to M$.
\end{definition}

Universal derivations exist and commute with base change and localization
\cite[Sec.~1.6]{huber2013etale}. In particular, for a morphism $f\colon X\to Y$
of adic spaces which is locally of topologically finite type we obtain a
cotangent sheaf $\Omega_{X/Y}^\an$, and it is a coherent $\sh{O}_{X}$-module.
For example,
$\Omega^\an_{K\langle T_1,\dots,T_n\rangle/K}$ is free on
$\dd T_1,\dots,\dd T_n$.

The importance of the cotangent complex is that it controls deformations.
For schemes this was conjectured by Grothendieck and proved by Illusie in his
thesis \cite{illusie1971thesis}. In the analytic setting we use the following
$p$-complete version, following \cite[\S7.1]{haoyang2019ht} (see also
\cite[\S7.2]{AlmostRingTheory}).
For a map $A_0 \to B_0$ of $p$-complete rings we write
$\what{\Cot}_{B_0/A_0} = \RR\lim_n (\Cot_{B_0/A_0} \otimes^\LL_{\bb{Z}} \bb{Z}/p^n)$
for the derived $p$-completion of the algebraic cotangent complex.

\begin{definition}
    [The analytic cotangent complex {\cite[\S7.1]{haoyang2019ht}}]
    \label{def::analytic_cotangent}
    Let $(A,A^+) \to (B,B^+)$ be a map of complete Tate--Huber pairs over
    $\bb{Q}_p$. The \emph{analytic cotangent complex} is
    \[
        \Cotan_{B/A} =
        \Big(\colim_{A_0 \to B_0} \what{\Cot}_{B_0/A_0}\Big)
        \left[ \frac{1}{p} \right],
    \]
    where the colimit runs over the filtered category of rings of definition
    $A_0 \subset A^+$ and $B_0 \subset B^+$ with $A_0\to B_0$.
    Sheafifying this construction gives, for a morphism $Y \to X$ of analytic
    adic spaces over $\Spa\bb{Q}_p$, an object
    $\Cotan_{Y/X} \in \icat{D}_{\geqq 0}(\sh{O}_{Y})$.
\end{definition}

\begin{remark}
    [{\cite[Rem.~7.1.1, 7.1.2]{haoyang2019ht}}] 
    \label{rmk::cotangent_simplification}
    In the cases of interest to us the colimit simplifies.
    If $A$ is topologically of finite type over a $p$-adic field, then one may
    take $A_0 = A^+$ and let only $B_0$ vary; if moreover $B^+$ is itself a ring
    of definition (e.g.\ $B$ affinoid perfectoid), then
    $\Cotan_{B/A} = \what{\Cot}_{B^+/A^+}[1/p]$.
\end{remark}

\begin{proposition}
    \label{prop::cotangent_properties}
    \label{prop::closed_immersions}
    Let $S$ be an analytic adic space over $\Spa\bb{Q}_p$ and let
    $f\colon X \to Y$ be a morphism of analytic adic spaces over $S$.
    \begin{enumerate}
        \item We have $\sh{H}_0(\Cotan_{X/Y}) = \Omega^\an_{X/Y}$.
        \item There is a natural \emph{transitivity fiber sequence}
            \[
                \LL f^* \Cotan_{Y/S} \to \Cotan_{X/S} \to \Cotan_{X/Y}
            \]
            in $\icat{D}(\sh{O}_X)$.
        \item If $X$, $Y$ are of topologically finite type over $B_2$ and $f$
            is a closed immersion with ideal $\sh{I}$, then
            $\tau_{\leqq 1}\Cotan_{X/Y} \reqv \sh{I}/\sh{I}^2[1]$;
            if $\sh{I}$ is locally generated by a regular sequence, then
            $\Cotan_{X/Y} \reqv \sh{I}/\sh{I}^2[1]$.
        \item If $X$, $Y$ are of topologically finite type over $B_2$ and $f$
            is smooth, then $\Cotan_{X/Y} \reqv \Omega^\an_{X/Y}$.
            In particular, if $f$ is \'etale, then $\Cotan_{X/Y} = 0$.
    \end{enumerate}
\end{proposition}

\begin{proof}
    For (1) and (2) see \cite[\S7.1]{haoyang2019ht}; the proof of
    \cite[Prop.~7.2.13]{AlmostRingTheory} applies verbatim, since derived
    $p$-completion, filtered colimits and inverting $p$ preserve fiber
    sequences. For (3), the map $X \to Y$ is finite, so the analytic and
    the algebraic cotangent complexes agree
    \cite[Prop.~5.2.15]{haoyang2021crystalline}, and the claim follows from
    the corresponding statement for the algebraic cotangent complex
    \cite[III.3.1.2]{illusie1971thesis}.
    For (4) see \cite[Cor.~5.2.14]{haoyang2021crystalline}.
\end{proof}

Finally we state the deformation-theoretic input, in the only case we need:
lifting a smooth rigid space $X$ over $C$ along the square-zero thickening
$\Spa C \inj \Spa B_2$, whose ideal is $C(1)$
(see Definition \ref{def::fontaine}).

\begin{proposition}
    [{\cite[\S7]{haoyang2019ht}, \cite[\S7.3]{AlmostRingTheory}}] 
    \label{prop::deformations}
    Let $X$ be a smooth rigid-analytic variety over $C$.
    There is a class
    \[
        o \in \Ext^2_{\sh{O}_X}\big(\Cotan_{X/C}, \sh{O}_X(1)\big)
          = \HH^2(X,\wtilde{T}_X)
    \]
    which vanishes if and only if $X$ admits a flat lift to $B_2$.
    In this case the isomorphism classes of lifts form a torsor under
    $\HH^1(X,\wtilde{T}_X)$, and the automorphism group of any lift is
    $\HH^0(X,\wtilde{T}_X)$.
\end{proposition}

Here the identification uses $\Cotan_{X/C} = \Omega^1_{X/C}$, by
Proposition~\ref{prop::cotangent_properties}(4), and
$\shHom(\Omega^1_{X/C},\sh{O}_X(1)) = \wtilde T_X$. The class $o$ is the
class of the transitivity sequence for $X \to \Spa C \to \Spa B_2$ considered
in \S\ref{sec::correspondence}.

\section{Conventions on diamonds}
\label{app::diamonds}
We recall the notions on diamonds that we use, all of which are taken from
\cite{scholze2017diamonds}; we refer there for the definitions of perfectoid
spaces, of (affinoid) pro-\'etale and quasi-pro-\'etale morphisms, of
$\vv$-covers and of diamonds \cite[Defs.~7.8, 8.1, 10.1, 11.1]{scholze2017diamonds}.
We denote by $\Perf$ the category of perfectoid spaces of characteristic $p$.

\begin{remark}
    In general, one might try to define pro-\'etale morphisms for general adic
    spaces using the pro-category of $X_\et$, as in \cite{scholze2013adic},
    which leads to the \emph{flattened pro-\'etale site} of a rigid space.
    However, some care with covers is needed there \cite{scholzeerratum}, and
    so we have opted to follow the approach of \cite{scholze2017diamonds} via
    quasi-pro-\'etale morphisms.
\end{remark}

Let $X$ be an analytic adic space over $\Spa\bb{Z}_p$ and $X^\diamond$ its
associated diamond \cite[Def.~15.5]{scholze2017diamonds}. Then
$|X^\diamond| = |X|$, the diamond $X^\diamond$ is locally spatial, and
$X \mapsto X^\diamond$ induces equivalences $X_\et \isoto X^\diamond_\et$ and
$X_\fet \isoto X^\diamond_\fet$ \cite[Lemma~15.6]{scholze2017diamonds}.
From now on we always see a rigid-analytic variety $X$ as a diamond, and
consider the following sites, in increasing order of fineness
\cite[Def.~14.1]{scholze2017diamonds}: 
\begin{itemize}
    \item the analytic site $X_\an$, associated to the topological space
        $|X|$;
    \item the (finite) \'etale site $X_\et$ ($X_\fet$) of (finite) \'etale
        maps $Y\to X$;
    \item the quasi-pro-\'etale site $X_\qproet$ of locally separated
        quasi-pro-\'etale maps $Y \to X$;
    \item the $\vv$-site $X_\vv$ of all small $\vv$-sheaves over $X$;
\end{itemize}
the last three with the $\vv$-topology. (To avoid set-theoretic issues one
fixes a cut-off cardinal as in \cite[\S4]{scholze2017diamonds}; this does not
affect cohomology, and we suppress it.)
There are natural maps of sites
\[
\begin{tikzcd}
    X_\vv \ar[r, "\lambda"] & X_\qproet \ar[r, "\nu"] & X_\et ,
\end{tikzcd}
\]
the pullbacks along $\lambda$ and $\nu$ are fully faithful, and the cohomology
of \'etale sheaves agrees on all three sites
\cite[Props.~14.7, 14.8]{scholze2017diamonds}.

Since $X^\diamond$ is quasi-pro-\'etale locally perfectoid, and perfectoid
spaces are pro-\'etale locally strictly totally disconnected
\cite[Lemma~7.18]{scholze2017diamonds}, and can moreover be taken
with extremally disconnected underlying space (so that all
quasi-pro-\'etale covers split),
the topos $X_\qproet$ is locally weakly contractible \cite[Def.~3.2.1]{proet}, and hence \emph{replete}
\cite[Prop.~3.2.3]{proet}. In particular $\RR\lim_n \sh{F}_n = \lim_n
\sh{F}_n$ for every tower of surjections of sheaves on $X_\qproet$
\cite[Prop.~3.1.10]{proet}.

\begin{remark}
    We mention the $\vv$-site mostly for completeness, as other references work
    with $\vv$-bundles. For vector bundles, $\VB(X_\vv) \simeq \VB(X_\qproet)$
    by \cite[Lemma~17.1.8]{scholze2020berkeley}; for perfect complexes, the
    analogous statement is \cite[Thm.~2.1]{anschutz-lebras2021fourier}.
    We therefore work with the quasi-pro-\'etale site throughout.
\end{remark}

\begin{definition}
    [The completed structure sheaf]
    \label{def::completed_structure_sheaf}
    Let $X$ be a rigid-analytic space over a non-archimedean field $K$.
    The \emph{completed structure sheaf} on $X_\qproet$ is
    \[
        \what{\sh{O}}_X^+ =
        \lim_n \nu^{-1}\sh{O}_{X}^+/p^n, \quad
        \what{\sh{O}}_X = \what{\sh{O}}_X^+ \left[ \frac{1}{p} \right].
    \]
    We also denote by $\underline{C}$ the sheaf $Y \mapsto \Hom_{\cont}(|Y|, C)$
    on $X_\qproet$, and often simply write $C$ for it.
\end{definition}

The cohomology of $\what{\sh{O}}_X$ satisfies the following acyclicity
property for affinoid perfectoids.

\begin{theorem}
    [{\cite[Prop.~8.5(iii)]{scholze2017diamonds}}]
    \label{thm::perfectoid_acyclicity}
    Let $X$ be a rigid-analytic space over $K$ and $Y = \Spa(R, R^+) \in
    X_\qproet$ an affinoid perfectoid.
    Then $\RR\Gamma(Y, \what{\sh{O}}_X) = R$.
\end{theorem}

In particular, if $Y \to X$ is a quasi-pro-\'etale cover of $X$ by an affinoid
perfectoid, then $\RR\Gamma(X, \what{\sh{O}}_X)$ is computed by the \v{C}ech
complex of this cover.

On the other hand, if $X$ is proper, then the cohomology of $\what{\sh{O}}_X$
computes the \'etale cohomology of $X$. By repleteness,
$\underline{\sh{O}_C} = \RR\lim_n \underline{\sh{O}_C}/p^n$, and hence for $X$
proper (using finiteness of $\RR\Gamma_\et(X,\bb{F}_p)$),
$\RR\Gamma(X_\qproet,\underline{C}) =
\RR\Gamma_\et(X,\bb{Z}_p)\otimes_{\bb{Z}_p} C$
is the \'etale cohomology of $X$ with $C$-coefficients.

\begin{theorem}
    [The generic version of the primitive comparison theorem]
    \label{thm::primitive_comparison}
    Let $X$ be a proper rigid-analytic variety over $C$. Then the natural map
    $C \to \what{\sh{O}}_X$ induces an equivalence of $\bb{E}_\infty$-algebras
    \[
        \RR\Gamma(X, C) \reqv \RR\Gamma(X, \what{\sh{O}}_X),
    \]
    where we take global sections on the quasi-pro-\'etale site\footnote{
        One can also use the $\vv$-site and the analogous structure sheaf
        (see \cite[Def.~2.1]{mannwerner2020local}). This is relevant since the
        proof below uses $\vv$-cohomology, but the difference is inessential
        \cite[Lemma~2.9]{mannwerner2020local}.
    }.
\end{theorem}

\begin{proof}
    By the primitive comparison theorem \cite[Cor.~3.9.24]{mann2022thesis},
    the map
    \[
        \RR\Gamma(X, \sh{O}_{C}/p) \to
        \RR\Gamma(X, \what{\sh{O}}_X^+/p)
    \]
    is an almost quasi-isomorphism. Using the exact sequences
    $0 \to p^n \sh{O}_{C}/p^{n+1} \to \sh{O}_{C}/p^{n+1}
    \to \sh{O}_{C}/p^{n} \to 0$
    and
    $0 \to p^n \what{\sh{O}}_X^+/p^{n+1} \to \what{\sh{O}}_X^+/p^{n+1}
        \to \what{\sh{O}}_X^+/p^n \to 0$,
    an induction on $n$ shows that
    $\RR\Gamma(X, \sh{O}_{C}/p^n) \to \RR\Gamma(X, \what{\sh{O}}_X^+/p^n)$
    is an almost quasi-isomorphism for all $n$.
    Since $X_\qproet$ is replete, we have
    $\sh{O}_{C} = \RR\lim_n \sh{O}_{C}/p^n$ and
    $\what{\sh{O}}_X^+ = \RR\lim_n \what{\sh{O}}_X^+/p^n$, and $\RR\Gamma$
    commutes with $\RR\lim$; hence
    $\RR\Gamma(X,\sh{O}_C) \to \RR\Gamma(X,\what{\sh{O}}_X^+)$ is an almost
    quasi-isomorphism, and the claim follows by inverting $p$.
\end{proof}

\begin{remark}
    This proof relies on Mann's $6$-functor formalism for $p$-torsion sheaves
    on diamonds \cite{mann2022thesis}. The original statement is due to
    Scholze \cite[Thm.~1.3]{scholze2013adic}, using the flattened pro-\'etale
    site.
\end{remark}

\printbibliography

\end{document}